\documentclass{amsart}

\usepackage{paralist}
\usepackage{url}
\usepackage{hyperref}

\newtheorem{theorem}{Theorem}[section]
\newtheorem{prop}[theorem]{Proposition}
\newtheorem{cor}[theorem]{Corollary}
\newtheorem{lemma}[theorem]{Lemma}

\newtheorem{remark}[theorem]{Remark}
\newtheorem{define}[theorem]{Definition}
\newtheorem{example}[theorem]{Example}

\newtheorem{assumption}[theorem]{Assumption}

\def\citet{\cite}
\newcommand{\comment}[1]{}

\newcommand{\RR}{{\mathbb R}}
\newcommand{\CC}{{\mathbb C}}

\begin{document}

\title[Minimizing Rational Functions by Exact Jacobian SDP Relaxation]
{Minimizing Rational Functions by Exact Jacobian SDP Relaxation
Applicable to Finite Singularities}
\author{Feng Guo}
\address{
Department of Mathematics, University of California, San Diego,
9500 Gilman Drive, La Jolla, CA 92093.}
\email{f1guo@math.ucsd.edu}
\author{Li Wang}
\address{
Department of Mathematics, University of California, San Diego
9500 Gilman Drive, La Jolla, CA 92093.} 
\email{liw022@math.ucsd.edu}
\author{Guangming Zhou}
\address{
School of Mathematics and Computing Science, Xiangtan University, Hunan 411105, P. R. China.} 
\email{zhougm@xtu.edu.cn}

\begin{abstract}
This paper considers the optimization problem of minimizing a rational
function. We reformulate this problem as polynomial optimization by the
technique of homogenization.  These two problems are shown to be equivalent
under some generic conditions. The exact Jacobian SDP relaxation method
proposed by Nie is used to solve the resulting polynomial
optimization. We also prove that the assumption of nonsingularity in
Nie's method can be weakened as the finiteness of singularities.
Some numerical examples are given to illustrate the efficiency of our method.
\end{abstract}

\maketitle

\section{Introduction}\label{sec::intro}

Consider the problem of minimizing a rational function  
\begin{equation}\label{pro::rf}
\left\{
\begin{aligned}
r^*:=\underset{x\in\RR^n}{\min}&\ {\frac{p(x)}{q(x)}}\\
\text{s.t.}&\ h_1(x)=\cdots=h_{m_1}(x)=0,\\
&\ g_1(x)\ge 0,\ldots,g_{m_2}(x)\ge 0.
\end{aligned}\right.
\end{equation}
where $p(x),q(x),h_i(x),g_j(x)\in\RR[x]:=\RR[x_1,\ldots,x_n]$. As a special
case, when $\deg{(q)}=0$, (\ref{pro::rf}) becomes a multivariate polynomial
optimization which is NP-hard even when $p(x)$ is a nonconvex quadratic
polynomial and $h_i(x)$'s, $g_j(x)$'s are linear \cite{PardalosVavasis1991}.

Some approaches using sum-of-squares relaxation to solve the minimization of
(\ref{pro::rf}) are proposed in \cite{JibeteanKlerkRF,Nie2008} and the core
idea therein is in the following. Let $S$ be the feasible set of
(\ref{pro::rf}). Suppose that $r^*>-\infty$ and $q(x)$ is nonnegative on $S$ 
(otherwise replace $\frac{p(x)}{q(x)}$ by $\frac{p(x)q(x)}{q^2(x)}$),
then $\gamma\in\RR$ is a lower bound of $r^*$ if and only if $p(x)-\gamma
q(x)\ge 0$ on $S$. Thus the problem (\ref{pro::rf}) can be reformulated as
maximizing $\gamma$ such that $p(x)-\gamma q(x)$ is nonnegative on $S$, which
is related to the representation of a nonnegative polynomial on a feasible set
defined by several polynomial equalities and inequalities.  As is well-known, a
univariate polynomial is nonnegative on $\RR$ if and only if it can be
represented as a sum-of-squares of polynomials (SOS)
\cite{Reznick96someconcrete} which can be efficiently determined by solving a
semidefinite program (SDP) \cite{Parrilo2003,parriloSturmfels}.  However, when
$n>1$, due to the fact that a nonnegative multivariate polynomial might not be
an SOS \cite{Reznick96someconcrete}, the problem (\ref{pro::rf}) becomes very
hard even if there are no constraints.  Denote by $M(S)$ the set of polynomials
which can be written as 
\[
\sum\limits_{i=1}^{m_1}\varphi_i(x)h_i(x)+\sum\limits_{j=0}^{m_2}\sigma_j(x)g_j(x)
\] 
where $\varphi_i(x)\in\RR[x]$, $g_0(x)=1$ and $\sigma_j(x)$'s are SOS.  $M(S)$
is called the {\itshape quadratic module} generated by the defining polynomials
of $S$.  If $M(S)$ is {\itshape archimedean}, which means that there exists
$R>0$ such that $R-||x||^2_2\in M(S)$, then Putinar's Positivstellensatz
\cite{Putinar1993} states that if a polynomial $f(x)$ is positive on $S$, then
it belongs to $M(S)$. If $S$ is compact, then Schm{\"{u}}dgen's
Positivstellensatz \cite{Schmugen1991} states that if a polynomial $f(x)$ is
positive on $S$, then it can be represented as
\begin{equation}\label{eq::preordering}
f(x)=\sum\limits_{i=1}^{m_1}\varphi_i(x)h_i(x)+
\sum\limits_{\nu\in\{0,1\}^{m_2}}\sigma_{\nu}(x)g_{\nu}(x),
\end{equation}
where $\varphi_i(x)\in\RR[x]$, $g_{\nu}(x)=g_1^{\nu_1}\cdots
g_{m_2}^{\nu_{m_2}}$ and $\sigma_{\nu}(x)$'s are SOS. The set of
polynomials which have representation (\ref{eq::preordering}) is called
{\itshape preordering} which we denote by $P(S)$.  Hence, if $S$ in
(\ref{pro::rf}) is archimedean or compact, we can apply Putinar's
Positivstellensatz or Schm{\"{u}}dgen's Positivstellensatz to maximize $\gamma$
such that $p(x)-\gamma q(x)$ belongs to $M(S)$ or $P(S)$. 

In this paper, we present a different way to obtain the minimum $r^*$. Given a
polynomial $f\in\RR[x]$, let $\tilde{x}=(x_0,x_1,\ldots,x_n)\in\RR^{n+1}$ and
$f^{hom}$ be the homogenization of $f$, i.e.
$f^{hom}(\tilde{x})=x_0^{\deg{(f)}}f(x/x_0)$.  We reformulate the minimization
of (\ref{pro::rf}) by the technique of homogenization as the following
polynomial optimization 
\begin{equation}\label{pro::homo}
\left\{
\begin{aligned}
s^*:=\underset{\tilde{x}\in\RR^{n+1}}{\min}&\ \tilde{p}(\tilde{x})\\
\text{s.t.}&\ h^{hom}_1(\tilde{x})=\cdots=h^{hom}_{m_1}(\tilde{x})=0,\ 
\tilde{q}(\tilde{x})=1,\\
&\ g^{hom}_1(\tilde{x})\ge 0,\ldots,g^{hom}_{m_2}(\tilde{x})\ge 0,\ 
x_0\ge 0,\\
\end{aligned}\right.
\end{equation}
where $\tilde{p}(\tilde{x}):=x_0^{\max\{\deg{(p)},\deg{(q)}\}}p(x/x_0)$ and
$\tilde{q}(\tilde{x}):=x_0^{\max\{\deg{(p)},\deg{(q)}\}}q(x/x_0)$. We show that
these two problems are equivalent under some generic conditions. As a special
case, they are always equivalent if there are no constraints in
(\ref{pro::rf}). The relations between the achievabilities of $r^*$ and $s^*$
are discussed. 

Therefore, the problem of solving (\ref{pro::rf}) becomes to efficiently
solving problem (\ref{pro::homo}).  For general polynomial optimization, there
has been much work on computing the infimum of the objective via SOS
relaxations, see the survey \cite{Laurent_sumsof} by Laurent and the references
therein.  A standard approach for solving polynomial optimizations is the
hierarchy of semidefinite programming relaxations proposed by Lasserre
\cite{LasserreGlobal2001}. Recently, under the assumption that the optimum is
achievable, some gradient type SOS relaxations are presented in
\cite{DNPKKT,Nie2006}.  When the optimum is an asymptotic value, we refer to
the approaches proposed in
\cite{GGSZ2011,GSZ2010,Markus06,Vui08,Vui-constraints}. However, to the best
knowledge of the authors, the finite convergence of the above methods is
unknown which means that we need to solve a big number of SDPs until the
convergence is met. More recently, Nie \cite{exactJacNie} proposed a new type
SDP relaxation using the minors of the Jacobian of the objective and
constraints.  It is shown \cite{exactJacNie} that the Jacobian SDP relaxation
is exact under some generic assumptions.  Therefore, in this paper we employ
the Jacobian SDP relaxation to solve (\ref{pro::homo}).

For another contribution of this paper, we prove that the assumptions under
which the Jacobian SDP relaxation \cite{exactJacNie} is exact can be weakened.
Let $J$ be the set containing polynomials in the equality constraints and an
arbitrary subset of the inequality constraints.  In order to prove the finite
convergence of the Jacobian SDP relaxation, it is assumed in \cite{exactJacNie}
that the Jacobian of polynomials in $J$ has full rank at any point in $V(J)$
which is the variety defined by $J$.  In other words, if the ideal $\langle
J\rangle$ generated by polynomials in $J$ is radical and its codimension equals
the number of these polynomials, the variety $V(J)$ needs to be nonsingular to
guarantee the finite convergence.  In this paper, we prove that the
nonsingularity in the assumptions can be replaced by the finiteness of
singularities.  More specifically, we show that if there are only finite 
points in $V(J)$ such that the Jacobian of polynomials in $J$ is a rank
deficient matrix, then the Jacobian SDP relaxation \cite{exactJacNie} is still
exact. We also give an example to illustrate the correctness of our result.

Another possible and natural reformulation of (\ref{pro::rf}) is
\begin{equation}\label{pro::y}
\left\{
\begin{aligned}
\bar{s}^*:=\min\limits_{x\in\RR^n,y\in\RR}&\ p(x)y\\
\text{s.t.}&\ h_1(x)=\cdots=h_{m_1}(x)=0,\ q(x)y=1,\\
&\ g_1(x)\ge 0,\ldots,g_{m_2}(x)\ge 0.\\
\end{aligned}\right.
\end{equation}
Clearly, if $r^*$ is achievable in (\ref{pro::rf}), then (\ref{pro::y}) is
equivalent to (\ref{pro::rf}) and we always have $r^*=\bar{s}^*$. One might ask
why we do not solve (\ref{pro::y}) instead of (\ref{pro::homo}). The reason is
that when we employ Jacobian SDP relaxation \cite{exactJacNie} to solve
(\ref{pro::homo}) or (\ref{pro::y}), we need to assume that the optimum is
achievable.  Actually, $s^*$ in (\ref{pro::homo}) is more likely to be
achievable than $\bar{s}^*$ in (\ref{pro::y}). To see this, note that when
$r^*$ is not achievable, $\bar{s}^*$ can not be reached either. However, $s^*$
might still be achievable when $r^*$ is not. Some sufficient conditions are
given in Theorem \ref{thm::main2} and they are not necessary (see Example
\ref{ex::commonzeros} and \ref{ex::commonzeros2}). For a simple example,
consider the problem 
\[
\min\limits_{x_1\in\RR}\ \frac{1}{x_1^2+1}.
\] 
Obviously, $r^*=\bar{s}^*=0$ and they are not achievable. 
However, we can reformulate it as
\begin{equation*}
\left\{
\begin{aligned}
s^*:=\min\limits_{x_0, x_1\in\RR}&\ x_0^2\\
\text{s.t.}&\ x_1^2+x_0^2=1.
\end{aligned}\right.
\end{equation*}
Then $s^*=0$ and we have two minimizers $(0,\pm 1)$ which verify that $r^*$ is
not achievable by (\ref{item::3}) in Theorem \ref{thm::main2}.
\vspace{5pt}

This paper is organized as follows. In Section \ref{sec::rf2po}, we reformulate
(\ref{pro::rf}) as (\ref{pro::homo}) by the technique of homogenization and
investigate the relations between the achievabilities of the optima of these
two optimizations.  We introduce the Jacobian SDP relaxation \cite{exactJacNie}
in Section \ref{sec::exactrelax} and show that the assumptions therein under
which the Jacobian SDP relaxation is exact can be weakened. In Section
\ref{sec::rf2}, we first return to solving the problem (\ref{pro::rf}) and make
some discussions, then we give some numerical examples to illustrate the
efficiency of our method. 
 
\vspace{8pt}

\paragraph{\bf Notation} The symbol ${\mathbb N}$ (resp., $\RR$, $\CC$) denotes the
set of nonnegative integers (resp., real numbers, complex numbers). For any
$t\in\RR$, $\lceil t\rceil$ denotes the smallest integer not smaller than $t$.
For integer $n>0$, $[n]$ denotes the set $\{1,\cdots,n\}$ and for a subset $J$
of $[n]$, $|J|$ denotes its cardinality. For $x\in\RR^n$, $x_i$ denotes the
$i$-th component of $x$. The symbol $\RR[x]=\RR[x_1,\ldots,x_n]$ (resp.,
$\CC[x]=\CC[x_1,\ldots,x_n]$) denotes the ring of polynomials in
$(x_1,\ldots,x_n)$ with real (resp. complex) coefficients. For
$\alpha\in{\mathbb N}^n$, denote $|\alpha|=\alpha_1+\cdots+\alpha_n$. For
$x\in\RR^n$ and $\alpha\in{\mathbb N}^n$, $x^{\alpha}$ denotes
$x_1^{\alpha_1}\cdots x_n^{\alpha_n}$. For a symmetric matrix $X$, $X\succeq 0$
(resp., $X\succ 0$) means $X$ is positive semidefinite (resp., positive
definite). For $u\in\RR^n$, $\|u\|_2$ denotes the standard Euclidean norm. 
${\mathcal C}^k$ denotes the class of functions whose $k$-th derivatives are continuous.

\section{Minimizing Rational Functions by Homogenization}\label{sec::rf2po}

In this section, we first reformulate the minimization of (\ref{pro::rf}) as
polynomial optimization (\ref{pro::homo}) by the technique of homogenization
and investigate the relations between the achievabilities of the optima of
these two problems. Then we show that the condition under which the problems
(\ref{pro::rf}) and (\ref{pro::homo}) are equivalent is generic. 

\subsection{Reformulating the minimization of rational functions by homogenization}
Given a polynomial $f\in\RR[x]$, let
$\tilde{x}=(x_0,x)=(x_0,x_1,\ldots,x_n)\in\RR^{n+1}$ and $f^{hom}(\tilde{x})$
be the homogenization of $f$, i.e.
$f^{hom}(\tilde{x})=x_0^{\deg{(f)}}f(x/x_0)$.  We define the following sets:
\begin{equation}\label{def::S}
\begin{aligned}
S&:=\left\{x\in\RR^n\mid h_i(x)=0,\ g_j(x)\ge 0,\ 
i\in[m_1],\ j\in[m_2]\right\},\\
\widetilde{S}_0&:=\left\{\tilde{x}\in\RR^{n+1}\mid h^{hom}_i(\tilde{x})=0,\ 
g_j^{hom}(\tilde{x})\ge 0,\ x_0>0,\ i\in[m_1],\ j\in[m_2]\right\},\\
\widetilde{S}&:=\left\{\tilde{x}\in\RR^{n+1}\mid h^{hom}_i(\tilde{x})=0,\ 
g_j^{hom}(\tilde{x})\ge 0,\ x_0\ge 0,\ i\in[m_1],\ j\in[m_2]\right\}.\\
\end{aligned}
\end{equation}
Recall that for integer $n>0$, $[n]$ denotes the set $\{1,\cdots,n\}$.  Let
$\text{closure}(\widetilde{S}_0)$ be the closure of $\widetilde{S}_0$ in
$\RR^{n+1}$.  From the above definition, we immediately have 
\begin{prop}\label{prop::eq}
$f(x)\ge 0$ on $S$ if and only if $f^{hom}(\tilde{x})\ge 0$ 
on $closure(\widetilde{S}_0)$.  
\end{prop}
\begin{proof}
We first prove the ``if'' part. Suppose $f^{hom}(\tilde{x})\ge 0$ on
$\text{closure}(\widetilde{S}_0)$.  If there exists a point $u\in S$ such that
$f(u)<0$, then $(1,u)\in \widetilde{S}_0$. Thus $f^{hom}(1,u)=f(u)<0$ which is
a contradiction. 

Next we prove the ``only if'' part. Suppose $f(x)\ge 0$ on $S$ and consider a
point $(u_0,u)\in\RR^{n+1}$ in the $\text{closure}(\widetilde{S}_0)$. There
exists a sequence $\{(u_{k,0},u_k)\}\in \widetilde{S}_0$ such that
$\lim\limits_{k\rightarrow\infty}(u_{k,0},u_k)=(u_0,u)$. Since $u_{k,0}>0$ for
all $k\in{\mathbb N}$, we consider the sequence $\{u_k/u_{k,0}\}$. For
$i=1,\ldots,m_1$ and $j=1,\ldots,m_2$, we have
$h_i(u_k/u_{k,0})=h_i^{hom}(u_{k,0},u_k)/(u_{k,0})^{\deg(h_i)}=0$ and
$g_j(u_k/u_{k,0})=g_j^{hom}(u_{k,0},u_k)/(u_{k,0})^{\deg(g_j)}\ge 0$.  It
implies that $\{u_k/u_{k,0}\}\in S$. Thus 
\[
f^{hom}(u_0,u)=\lim\limits_{k\rightarrow\infty}f^{hom}(u_{k,0},u_k)
=\lim\limits_{k\rightarrow\infty}u_{k,0}^{\deg{(f)}}f(u_k/u_{k,0})\ge 0,
\]
which concludes the proof.
\end{proof}

\vspace{5pt}

Let $d=\max\{\deg{(p)},\deg{(q)}\}$, $\tilde{p}(\tilde{x})=x_0^dp(x/x_0)$ and
$\tilde{q}(\tilde{x})=x_0^dq(x/x_0)$. We reformulate the minimization of
(\ref{pro::rf}) as the following constrained polynomial optimization: 
\begin{equation*}
\left\{
\begin{aligned}
s^*:=\underset{\tilde{x}\in\RR^{n+1}}{\min}&\ \tilde{p}(\tilde{x})\\
\text{s.t.}&\ h^{hom}_1(\tilde{x})=\cdots=h^{hom}_{m_1}(\tilde{x})=0,\ 
\tilde{q}(\tilde{x})=1,\\
&\ g^{hom}_1(\tilde{x})\ge 0,\ldots,g^{hom}_{m_2}(\tilde{x})\ge 0,\ 
x_0\ge 0.\\
\end{aligned}\right.
\end{equation*}
We now investigate the relations between $r^*$ and $s^*$.  In the following of
this paper, without loss of generality, we always assume that
\begin{equation}\label{assump::q>0}
\begin{aligned} 
&q(x)>0\ \text{on a neighbourhood of a minimizer of}\
(\ref{pro::rf});\ \text{or}\ q(x)>0\ \text{for}\\  
&\text{any}\ x\in S\ \text{with sufficient large  Euclidean norm if}\ 
r^*\ \text{is not achievable}. 
\end{aligned}
\end{equation} 
Otherwise we can replace $\frac{p(x)}{q(x)}$ by $\frac{p(x)q(x)}{q^2(x)}$.
Note that we do not assume $q(x)$ is nonnegative on the whole feasible set $S$
as in \cite{JibeteanKlerkRF,Nie2008}. 
\begin{define}{\upshape\cite{exactJacNie}}
If there exists a point $0\neq (0,u)\in\widetilde{S}$ but $(0,u)\notin
closure(\widetilde{S}_0)$, then we say $S$ is not {\itshape closed at
$\infty$}; otherwise, we say $S$ is closed at $\infty$.
\end{define}
\begin{theorem}\label{thm::main}
It always holds that $s^*\le r^*$, and the equality holds if one of the
following conditions is satisfied:
\begin{enumerate}[\upshape (a)]
\item $S$ is closed at $\infty$; 
\item $\deg(p)>\deg(q)$;
\item $s^*$ is achievable and $x^*_0>0$
for at least one of its minimizers $\tilde{x}^*=(x^*_0,x^*)$.
\end{enumerate} 
\end{theorem}
\begin{proof}
We first show that $s^*\le r^*$. For any $u\in S$ in a neighborhood of a
minimizer of (\ref{pro::rf}) or with sufficient large Euclidean norm if $r^*$
is not achievable, if $\frac{p(x)}{q(x)}$ is defined at $u$, then $q(u)>0$ by
the assumption in (\ref{assump::q>0}). Let $t=q(u)^{1/d}=\tilde{q}(1,u)^{1/d}$.
We have $\tilde{q}(1/t,u/t)=1$ and $(1/t,u/t)\in \widetilde{S}$, so 
\[
\frac{p(u)}{q(u)}=\frac{\tilde{p}(1,u)}{\tilde{q}(1,u)}=
\frac{\tilde{p}(1/t,u/t)}{\tilde{q}(1/t,u/t)}=\tilde{p}(1/t,u/t)\ge s^*,
\]
then we have $s^*\le r^*$. Therefore, to show $r^*=s^*$, we only need to show
$r^*\le s^*$.

(a) For any feasible point $(u_0,u)$ of (\ref{pro::homo}), i.e., $(u_0,u)\in
\widetilde{S}$ and $\tilde{q}(u_0,u)=1$, since $S$ is closed at $\infty$, there
exists a sequence $\{(u_{k,0},u_k)\}$ in $\widetilde{S}$ such that $u_{k,0}>0$
for any $k\in{\mathbb N}$ and
$\underset{k\rightarrow\infty}{\lim}(u_{k,0},u_k)=(u_0,u)$.  Due to the
continuity of $\tilde{q}$,
$\lim\limits_{k\rightarrow\infty}\tilde{q}(u_{k,0},u_k)=1$. Hence, we can
always assume that for any $k\in{\mathbb N}$, $\tilde{q}(u_{k,0},u_k)>0$.  For
each $k\in{\mathbb N}$, let $t_k=\tilde{q}(u_{k,0},u_k)^{1/d}$ and consider the
sequence $\{(u_{k,0}/t_k,u_k/t_k)\}$. We have
$\underset{k\rightarrow\infty}{\lim}(u_{k,0}/t_k,u_k/t_k)=(u_0,u_k)$ and
$\tilde{q}(u_{k,0}/t_k,u_k/t_k)=1$. For $i=1,\ldots,{m_1}$, $j=1,\ldots,{m_2}$,
\begin{equation*}
\begin{aligned}
&0=\frac{1}{t_k^{\deg{(h_i)}}}h^{hom}_i(u_{k,0},u_k)=h^{hom}_i(u_{k,0}/t_k,u_k/t_k)=
\frac{1}{t_k^{\deg{(h_i)}}}u_{k,0}^{\deg{(h_i)}}h_i(u_k/u_{k,0}),\\
&0\le \frac{1}{t_k^{\deg{(g_j)}}}g^{hom}_j(u_{k,0},u_k)=g^{hom}_j(u_{k,0}/t_k,u_k/t_k)=
\frac{1}{t_k^{\deg{(g_j)}}}u_{k,0}^{\deg{(g_j)}}g_j(u_k/u_{k,0}),
\end{aligned}
\end{equation*}
which imply $(u_{k,0}/t_k,u_k/t_k)\in \widetilde{S}$ and $u_k/u_{k,0}\in S$
for all $k$.  Hence 
\[
\tilde{p}(u_{k,0}/t_k,u_k/t_k)=\frac{\tilde{p}(u_{k,0}/t_k,u_k/t_k)}
{\tilde{q}(u_{k,0}/t_k,u_k/t_k)}=\frac{p(u_k/u_{k,0})}{q(u_k/u_{k,0})}\ge r^*
\]
and
$\tilde{p}(u_0,u)=\lim\limits_{k\rightarrow\infty}\tilde{p}(u_{k,0}/t_k,u_k/t_k)\ge
r^*$ which means $r^*\le s^*$.

(b) If $\deg(p)>\deg(q)$, then $x_0$ divides $\tilde{q}(\tilde{x})$. By
$\tilde{q}(\tilde{x})=1$, we have $u_0>0$ for any feasible point $(u_0,u)$ of
(\ref{pro::homo}) and it is easy to see that $u/u_0\in S$, then 
\[
\tilde{p}(u_0,u)=\frac{\tilde{p}(u_0,u)}{\tilde{q}(u_0,u)}
=\frac{\tilde{p}(1,u/u_0)}{\tilde{q}(1,u/u_0)}=\frac{p(u/u_0)}{q(u/u_0)}\ge
r^*,
\]
which means $r^*\le s^*$.

(c) Since $x_0^*>0$, we have $x^*/x^*_0\in S$ and 
\[
s^*=\tilde{p}(x_0^*,x^*)=\frac{\tilde{p}(x_0^*,x^*)}
{\tilde{q}(x_0^*,x^*)}=\frac{p(x^*/x_0^*)}{q(x^*/x_0^*)}\ge r^*,
\]
which implies $r^*=s^*$.
\end{proof}

The following corollary shows that the minimizations of (\ref{pro::rf}) and
(\ref{pro::homo}) are always equivalent when there are no constraints in
(\ref{pro::rf}).
\begin{cor}\label{cor::eq}
If ${m_1}={m_2}=0$ in {\upshape(\ref{pro::rf})}, then $S=\RR^n$ is closed at
$\infty$ and $r^*=s^*$.
\end{cor}

\begin{remark}
If $S=\RR^n$, we can remove $x_0\ge 0$ in {\upshape(\ref{pro::homo})}. In fact, if
there are no constraints, according to the proof of Part $(a)$ in Theorem
\ref{thm::main}, we only need $u_{k,0}\neq 0$ to get the same result.
Therefore, the global minimization
\[
r^*:=\underset{x\in\RR^n}{\min}\frac{p(x)}{q(x)}
\] 
is equivalent to 
\begin{equation}\label{pro::homo::global}
\left\{
\begin{aligned}
s^*:=\underset{\tilde{x}\in\RR^{n+1}}{\min}&\ \tilde{p}(\tilde{x})\\
\text{s.t.}&\ \tilde{q}(\tilde{x})=1.\\
\end{aligned}\right.
\end{equation}
\end{remark}
\begin{remark}\label{rem::notclosed}
We would like to point out that not every $S$ is closed at $\infty$ and $s^*$
might be strictly smaller than $r^*$ in this case. For example, consider the
following problem:
\begin{equation}\label{pro::conex}
\left\{
\begin{aligned}
r^*:=\underset{x_1,x_2\in\RR}{\min}&\ \frac{x_1}{(x_1-x_2)^2}\\
\text{s.t.}&\ x_1^2(x_1-x_2)=1,\\
&\ x_1-1\ge 0.
\end{aligned}\right.
\end{equation}
Clearly, we have $r^*=1$. However, \cite[Example 5.2 (i)]{NieDisNon} shows that
the set 
\[
\{(x_1,x_2)\in\RR^2\mid x_1^2(x_1-x_2)-1=0\}
\]
is not closed at $\infty$. Actually, 
\[
S:=\{(x_1,x_2)\in\RR^2\mid x_1^2(x_1-x_2)-1=0, x_1-1\ge 0\}
\]
is not closed at $\infty$, either.  To see it, we have 
\[
\widetilde{S}:=\{(x_0,x_1,x_2)\in\RR^3\mid x_1^2(x_1-x_2)-x_0^3=0, x_1-x_0\ge 0, 
 x_0\ge 0\}.
\]
Consider the point $(0,0,1)\in\widetilde{S}$.  Suppose that there exists a
sequence $\{(x_{k,0},x_{k,1},x_{k,2})\}$ in $\widetilde{S}$ such that
$\underset{k\rightarrow\infty}{\lim}(x_{k,0},x_{k,1},x_{k,2})=(0,0,1)$ and
$x_{k,0}>0$ for all $k\in{\mathbb N}$. Then for $0<\varepsilon<1/2$, there
exists $N\in{\mathbb N}$ such that for any $k>N$, we have 
\[
0<x_{k,0}<\varepsilon,\ \vert x_{k,1}\vert<\varepsilon,\ \vert x_{k,2}-1\vert<\varepsilon.
\]
Thus
\[
0<x_{k,0}^3=x_{k,1}^2(x_{k,1}-x_{k,2})\le x_{k,1}^2(\varepsilon-1+\varepsilon)\le 0
\]
which is a contradiction. Therefore, $S$ is not closed at $\infty$ and we have
$s^*=0< r^*$ if we reformulate {\upshape(\ref{pro::conex})} by homogenization
as the following problem
\begin{equation*}
\left\{
\begin{aligned}
s^*:=\underset{x_0,x_1,x_2\in\RR}{\min}&\ x_0x_1\\
\text{s.t.}&\ (x_1-x_2)^2-1=x_1^2(x_1-x_2)-x_0^3=0,\\
&\ x_1-x_0\ge 0, x_0\ge 0.
\end{aligned}\right.
\end{equation*}
However, in section \ref{subsec::generic} we will show that the closedness at
$\infty$ is a generic condition for a given set $S$.
\end{remark}
Let 
\[
\widehat{S}:=\{
x\in\RR^{n}\mid 
\hat{h}_i(x)=0,\ \hat{g}_j(x)\ge 0,\ i=1,\ldots,{m_1},\ j=1,\ldots,{m_2}\}
\]
where $\hat{h}_i$ and $\hat{g}_j$ denote the homogeneous parts of the highest
degree of $h_i$ and $g_j$, respectively.  Denote $p_d(x)$ and $q_d(x)$ the
homogeneous parts of degree $d$ of $p(x)$ and $q(x)$, respectively.
\begin{theorem}\label{thm::main2}
If one of the conditions in Theorem \ref{thm::main} holds, then the following
properties hold.
\begin{enumerate}[\upshape (a)]
\item\label{item::1} $r^*$ is achievable if and only if $s^*$ is achievable at
a minimizer $\tilde{x}^*=(x_0^*,x^*)$ with $x_0^*\neq 0$;
\item\label{item::2} If neither $p(x)$ and $q(x)$ have real common roots in $S$,
nor $p_d(x)$ and $q_d(x)$ have real nonzero common roots in $\widehat{S}$, 
then $s^*$ is achievable. 
\item\label{item::3} If $s^*$ is achievable and $x_0^*=0$ for all minimizers
$\tilde{x}^*=(x_0^*,x^*)$ of {\upshape(\ref{pro::homo})}, then $r^*$ is not
achievable.  For each minimizer $\tilde{x}^*=(0,x^*)$ of
{\upshape(\ref{pro::homo})}, if there exists a sequence
$\{\tilde{x}_k\}=\{(x_{k,0},x_k)\}$ in $\widetilde{S}$ such that
$\underset{k\rightarrow\infty}{\lim}\tilde{x}_k=\tilde{x}^*$ and $x_{k,0}>0$
for all $k\in{\mathbb N}$, then
$\underset{k\rightarrow\infty}{\lim}\frac{p(x_k/x_{k,0})}{q(x_k/x_{k,0})}=r^*$.
\end{enumerate}
\end{theorem}
\begin{proof}
If one of the conditions in Theorem \ref{thm::main} holds, we have $r^*=s^*$.

(a) Let $x^*$ be a minimizer of (\ref{pro::rf}), then $x^*\in S$ and
$t=\tilde{q}(1,x^*)^{1/d}=q(x^*)^{1/d}>0$ by the assumption in
(\ref{assump::q>0}).  It is easy to verify that $(1/t,x^*/t)\in \widetilde{S}$
and $\tilde{q}(1/t,x^*/t)=1$.  We have $\tilde{p}(1/t,x^*/t)=r^*=s^*$ which
means $(1/t,x^*/t)$ is a minimizer of (\ref{pro::homo}). If $s^*$ is achieved
at $\tilde{x}^*=(x_0^*,x^*)\in \widetilde{S}$ with $x_0^*>0$, then $r^*$ is
achieved at $x^*/x_0^*\in S$.

(b) To the contrary, we assume that $s^*$ is not achievable. Then there exists
a sequence $\{\tilde{x}_k\}$ in $\widetilde{S}$ such that
$\underset{k\rightarrow\infty}{\lim}{||\tilde{x}_k||}_2=\infty$,
$\underset{k\rightarrow\infty}{\lim}\tilde{p}(\tilde{x}_k)=s^*$ and for all
$k\in{\mathbb N}$, $\tilde{q}(\tilde{x}_k)=1$. Consider the bounded sequence
$\{\tilde{x}_k/{||\tilde{x}_k||}_2\}\subseteq \widetilde{S}$. By
Bolzano-Weierstrass Theorem, there exists a subsequence
$\{\tilde{x}_{k_j}/{||\tilde{x}_{k_j}||}_2\}$ such that
$\underset{j\rightarrow\infty}{\lim}\tilde{x}_{k_j}/{||\tilde{x}_{k_j}||}_2=\tilde{y}$
for some nonzero $\tilde{y}=(y_0,y)\in \widetilde{S}$ since $\widetilde{S}$ is
closed.  Let $\tilde{p}(\tilde{x}_{k_j})=s_{k_j}$, then
$\underset{j\rightarrow\infty}{\lim}{s_{k_j}}=s^*$. Since
$\tilde{p}(\tilde{x}_{k_j})=({||x_{k_j}||}_2)^d\tilde{p}(\tilde{x}_{k_j}/{||\tilde{x}_{k_j}||}_2)$
and $\underset{j\rightarrow\infty}{\lim}{||\tilde{x}_{k_j}||}_2=\infty$, we
have
$\tilde{p}(\tilde{y})=\underset{j\rightarrow\infty}{\lim}\tilde{p}(\tilde{x}_{k_j}/{||\tilde{x}_{k_j}||}_2)=0$.
Similarly, we can prove
$\tilde{q}(\tilde{y})=\underset{j\rightarrow\infty}{\lim}\tilde{q}(\tilde{x}_{k_j}/{||\tilde{x}_{k_j}||}_2)=0$.
Thus $\tilde{p}(\tilde{x})$ and $\tilde{q}(\tilde{x})$ have real nonzero common
root $\tilde{y}$ on unit sphere $S^{n+1}$.  We have $y_0=0$, otherwise $y/y_0$
is a real common root of $p(x)$ and $q(x)$ in $S$.  Therefore
$0=\tilde{p}(\tilde{y})=p_d(y),\ 0=\tilde{q}(\tilde{y})=q_d(y)$,\
$0=h^{hom}_i(\tilde{y})=\hat{h}_i(y)$ and $0\le
g^{hom}_j(\tilde{y})=\hat{g}_j(y)$, i.e. $p_d(x)$ and $q_d(x)$ have real
nonzero common root $y$ in $\widehat{S}$ which is a contradiction.

(c) By (\ref{item::1}), if $x_0^*=0$ for all minimizers of (\ref{pro::homo}),
$r^*$ is not achievable. Suppose $\tilde{x}^*= (0,x^*)$ is a minimizer of
(\ref{pro::homo}) and there exists a sequence
$\{\tilde{x}_k\}=\{(x_{k,0},x_k)\}$ in $\widetilde{S}$ such that
$\underset{k\rightarrow\infty}{\lim}\tilde{x}_k=\tilde{x}^*$ and $x_{k,0}>0$.
Then for each $k\in{\mathbb N}$, $x_k/x_{k,0}\in S$.  Since $\tilde{p}$ and
$\tilde{q}$ are continuous,
$\lim\limits_{k\rightarrow\infty}\tilde{p}(x_{k,0},x_k)=s^*$ and
$\lim\limits_{k\rightarrow\infty}\tilde{q}(x_{k,0},x_k)=1$.  Therefore, 
\[
\lim\limits_{k\rightarrow\infty}\frac{p(x_k/x_{k,0})}{q(x_k/x_{k,0})}
=\lim\limits_{k\rightarrow\infty}\frac{\tilde{p}(x_{k,0},x_k)}{q(x_{k,0},x_k)}
=s^*=r^*.
\]
Here completes the proof.
\end{proof}

\subsection{On the generality of closedness at infinity}\label{subsec::generic}
Although we have counter example in Remark \ref{rem::notclosed}, we next show 
that in general a given set $S$ in (\ref{def::S}) is indeed closed at $\infty$.
Therefore, if the constraints in (\ref{pro::rf}) are generic, (\ref{pro::rf}) 
and (\ref{pro::homo}) are equivalent.

Let us first review some elementary background about {\itshape resultants} and
{\itshape discriminants}.  More details can be found in
\cite{Cox-Little-OShea:UAG2005,DRMD,NieDisNon}. Let $f_1,\ldots,f_n$ be
homogeneous polynomials in $x=(x_1,\ldots,x_n)$. The resultant
$Res(f_1,\ldots,f_n)$ is a polynomial in the coefficients of $f_1,\ldots,f_n$
satisfying 
\[
Res(f_1,\ldots,f_n)=0\quad\Leftrightarrow\quad\exists 0\neq u\in\CC^n, 
\ f_1(u)=\cdots=f_n(u)=0.
\]
Let $f_1,\ldots,f_m$ be homogeneous polynomials with $m<n$.
The discriminant for $f_1,\ldots,f_m$, denoted by $\Delta(f_1,\ldots,f_m)$, is 
a polynomial in the coefficients of $f_1,\ldots,f_m$ such that
\[
\Delta(f_1,\ldots,f_m)=0
\]
if and only if the polynomial system
\[
f_1(x)=\cdots=f_m(x)=0
\]
has a solution $0\neq u\in\CC^n$ such that the Jacobian matrix of
$f_1,\ldots,f_n$ does not have full rank.

We next show that in general a given set $S$ in (\ref{def::S}) is closed at
$\infty$.  In the following, we suppose $S$ is not closed at $\infty$ and fix a
nonzero point $(0,u)\in\widetilde{S}$ but
$(0,u)\notin{\text{closure}}(\widetilde{S}_0)$.  Let $J(u):=\{j\in[m_2]\mid
g^{hom}_j(0,u)=0\}$. Then $g_j^{hom}(0,u)>0$ for all $j\in [m_2]\backslash
J(u)$.  We have the cardinality $m_1+\vert J(u)\vert\ge 1$, otherwise $(0,u)$
is an interior point of $\widetilde{S}$ and
$(0,u)\in{\text{closure}}(\widetilde{S}_0)$. Let
\[
V(u):=\{\tilde{x}\in\RR^{n+1}\mid h_i^{hom}(\tilde{x})=0,\
g_j^{hom}(\tilde{x})=0,\ i\in [m_1],\ j\in J(u)\}.
\] 
For any $\delta>0$, let
\[
B((0,u),\delta)=\{(x_0,x)\in\RR^{n+1}\mid\Vert(x_0,x)-(0,u)\Vert_2\le\delta\}.
\]
\begin{lemma}\label{lem::delta}
Suppose $S$ is not closed at $\infty$, then there exists $\delta>0$ such that
for all $(x_0,x)\in B((0,u),\delta)\cap V(u)$, we have $x_0\le 0$.
\end{lemma}
\begin{proof}
Suppose such $\delta$ doesn't exist. Consider a sequence $\{\delta_k\}$ with
$\delta_k>0$ and $\lim\limits_{k\rightarrow\infty}\delta_k=0$. Then for each
$k$, there exists a point $(u_{k,0},u_k)\in B((0,u),\delta_k)\cap V(u)$ such that
$u_{k,0}>0$. By the continuity, there exists $N$ such that for all $k\ge N$,
$g_j^{hom}(u_{k,0},u_k)>0$ for each $j\in [m_2]\backslash J(u)$ which implies
$(u_{k,0},u_k)\in\widetilde{S}$ for all $k\ge N$ and
$(0,u)\in{\text{closure}}(\widetilde{S}_0)$. The contradiction follows.
\end{proof}
Now let us recall the Implicit Function Theorem.
\begin{theorem}\cite[Theorem 3.3.1; {\scshape The Implicit Function Theorem}]{ImplicitFunctionTh}
\label{th::implicitfunctionth}
Let 
\[
\Phi(x)=\Phi(x_1,\ldots,x_n)\equiv(\phi_1(x_1,\ldots,x_n),\ldots,\phi_m(x_1,\ldots,x_n))
\]
be a mapping of class ${\mathcal C}^k$, $k\ge 1$, defined on an open set
$U\subseteq\RR^n$ and taking values in $\RR^m$. Assume that $1\le m<n$. Let
$x^0=(x_1^0,\ldots,x_n^0)$ be a fixed point of $U$ and
$x_a^0=(x_1^0,\ldots,x_{n-m}^0)$. Suppose that the Jacobian determinant 
\[
\frac{\partial(\phi_1,\ldots,\phi_m)}{\partial(x_{n-m+1},\ldots,x_n)}(x^0)\neq 0.
\]
Then there exists a neighborhood $\widetilde{U}$ of $x^0$, and open set
$W\subseteq\RR^{n-m}$ containing $x_a^0$, and functions $f_1,\ldots,f_m$ of
class ${\mathcal C}^k$ on $W$ such that 
\[
\Phi(x_1,\ldots,x_{n-m},f_1(x_a),\ldots,f_m(x_a))=0\quad \text{for every}\ x_a\in W.
\]
Here, $x_a=(x_1,\ldots,x_{n-m})$. Furthermore, $f_1,\ldots,f_m$ are the unique functions
satisfying 
\[
\left\{x\in\widetilde{U}\mid\Phi(x)=0\right\}=\left\{x\in\widetilde{U}\mid x_a\in W, x_{n-m+k}=f_k(x_a)\ 
\text{for}\ k=1,\ldots,m\right\}.
\]
\end{theorem}
Let $J(u)=\{j_1,\ldots,j_l\}$ and
\[
A(u):=\left[\begin{array}{ccc}
\frac{\partial h_1^{hom}}{\partial x_1}(0,u) & \cdots & \frac{\partial h_1^{hom}}{\partial x_n}(0,u)\\
\vdots & \vdots & \vdots\\
\frac{\partial h_{m_1}^{hom}}{\partial x_1}(0,u) & \cdots & \frac{\partial h_{m_1}^{hom}}{\partial x_n}(0,u)\\
\frac{\partial g_{j_1}^{hom}}{\partial x_1}(0,u) & \cdots & \frac{\partial g_{j_1}^{hom}}{\partial x_n}(0,u)\\
\vdots & \vdots & \vdots\\
\frac{\partial g_{j_l}^{hom}}{\partial x_1}(0,u) & \cdots & \frac{\partial g_{j_l}^{hom}}{\partial x_n}(0,u)\\
\end{array}\right]
=\left[\begin{array}{ccc}
\frac{\partial \hat{h}_1}{\partial x_1}(u) & \cdots & \frac{\partial \hat{h}_1}{\partial x_n}(u)\\
\vdots & \vdots & \vdots\\
\frac{\partial \hat{h}_{m_1}}{\partial x_1}(u) & \cdots & \frac{\partial \hat{h}_{m_1}}{\partial x_n}(u)\\
\frac{\partial \hat{g}_{j_1}}{\partial x_1}(u) & \cdots & \frac{\partial \hat{g}_{j_1}}{\partial x_n}(u)\\
\vdots & \vdots & \vdots\\
\frac{\partial \hat{g}_{j_l}}{\partial x_1}(u) & \cdots & \frac{\partial \hat{g}_{j_l}}{\partial x_n}(u)\\
\end{array}\right]
\]
Recall that $\hat{h}_i$ and $\hat{g}_j$ denote the homogeneous parts of the
highest degree of $h_i$ and $g_j$, respectively. Combining Lemma
\ref{lem::delta} and the Implicit Function Theorem, we have 
\begin{lemma}\label{lem::rank}
Suppose $S$ is not closed at $\infty$ and $m_1+\vert J(u)\vert<n+1$, then
$\text{rank}\ A(u)<m_1+\vert J(u)\vert$.
\end{lemma}
\begin{proof}
Let $m=m_1+\vert J(u)\vert$. Suppose $\text{rank}\ A(u)=m$.  Then there exist
$m$ independent columns in $A(u)$. Without loss of generality, we assume the
last $m$ columns of $A(u)$ are independent, i.e. the Jacobian determinant
\[
\frac{\partial(h^{hom}_1,\ldots,h^{hom}_{m_1},g_{j_1}^{hom},\ldots,g_{j_l}^{hom})}{
\partial(x_{n-m+1},\ldots,x_n)}(0,u)\neq 0.
\]
Partition $\tilde{u}=(0,u)$ as $(\tilde{u}^a,\tilde{u}^b)$ where 
$\tilde{u}^a=(0,u_1,\ldots,u_{n-m}), \tilde{u}^b=(u_{n-m+1},\ldots,u_n)$.
Then by the Implicit Function Theorem \ref{th::implicitfunctionth}, 
there exists an open set $W\subseteq \RR^{n-m+1}$ containing $\tilde{u}^a$, 
and functions $f_1,\ldots,f_m$ of class ${\mathcal C}^k$ on $W$ such that
\begin{equation*}
\begin{aligned}
&h^{hom}_i(x_0,\ldots,x_{n-m},f_1(\tilde{x}^a),\ldots,f_m(\tilde{x}^a))=0,\ i=1,\ldots,m_1,\\
&g^{hom}_{j}(x_0,\ldots,x_{n-m},f_1(\tilde{x}^a),\ldots,f_m(\tilde{x}^a))=0,\ j\in J(u),\\
\end{aligned}
\end{equation*}
for every $\tilde{x}^a\in W$. Here, $\tilde{x}^a=(x_0,\ldots,x_{n-m})$.
Therefore, $(\tilde{x}^a,f_1(\tilde{x}^a),\ldots,f_m(\tilde{x}^a))\in V(u)$ for
every $\tilde{x}^a\in W$. Since $W$ is open and $f_1,\ldots,f_m$ are
continuous, we can choose $\tilde{x}^a$ very close to $\tilde{u}^a$ such that
$(\tilde{x}^a,f_1(\tilde{x}^a),\ldots,f_m(\tilde{x}^a))\in B((0,u),\delta)\cap
V(u)$ with $x_0>0$ for any $\delta>0$, which contradicts the conclusion in
Lemma \ref{lem::delta}.
\end{proof}
The following theorem shows that if the defining polynomials of $S$ are generic, 
then $S$ is closed at $\infty$.
\begin{theorem}\label{th}
Suppose $S$ is not closed at $\infty$, then
\begin{enumerate}[\upshape (a)]
\item if $m_1+\vert J(u)\vert\ge n+1$, then $\text{Res}(h^{hom}_1,\ldots,
h^{hom}_{m_1},g_{j_1}^{hom},\ldots,g_{j_{n-m_1+1}}^{hom})=0$ for
every $\{j_1,\ldots,j_{n-m_1+1}\}\subseteq J(u)$;
\item if $m_1+\vert J(u)\vert=n$, then 
$\text{Res}(\hat{h}_1,\ldots,\hat{h}_{m_1},\hat{g}_{j_1},\ldots,\hat{g}_{j_l})=0$;
\item if $m_1+\vert J(u)\vert<n$, then 
$\Delta(\hat{h}_1,\ldots,\hat{h}_{m_1},\hat{g}_{j_1},\ldots,\hat{g}_{j_l})=0$.
\end{enumerate}
\end{theorem}
\begin{proof}
Since $h_i^{hom}(0,u)=\hat{h}_i(u)=0,\ g_j^{hom}(0,u)=\hat{g}_j(u)=0$ for all
$i\in[m_1], j\in J(u)$, then the conclusions in (a) and (b) are implied by the
proposition of resultants. If $m_1+\vert J(u)\vert<n$, then by Lemma
\ref{lem::rank}, the Jacobian matrix of
$(\hat{h}_1,\ldots,\hat{h}_{m_1},\hat{g}_{j_1},\ldots\\
,\hat{g}_{j_l})$ does not
have full rank at $u$. Hence, the conclusion in (c) follows by the proposition
of discriminants.
\end{proof}

In this section, we reformulate the minimization of (\ref{pro::rf}) as the
polynomial optimization (\ref{pro::homo}) by homogenization.  Suppose $S$ is
closed at $\infty$ which is generic and always true when $S=\RR^n$, then
$r^*=s^*$. The relations between the achievabilities of $r^*$ and $s^*$ are
discussed in Proposition \ref{thm::main2}.  Now the problem becomes how to
efficiently solve polynomial optimization (\ref{pro::homo}).  Recently, there
has been much work on solving polynomial optimization with or without
constraints via SOS relaxation.  In next section, we introduce the Jacobian SDP
relaxation \cite{exactJacNie} and show that the assumptions under which the
Jacobian SDP relaxation is exact can be weakened.

\section{Jacobian SDP Relaxation Applicable to Finite Real
Singularities}\label{sec::exactrelax}

Consider the following polynomial optimization problem
\begin{equation}\label{pro::opti}
\left\{
\begin{aligned}
f_{min}:=\underset{x\in\RR^n}{\min}&\ f(x)\\
\text{s.t.}&\ h_1(x)=\cdots=h_{m_1}(x)=0,\\
&\ g_1(x)\ge 0, \ldots, g_{m_2}(x)\ge 0.
\end{aligned}\right.
\end{equation}
where $f(x),h_i(x),g_j(x)\in\RR[x_1,\ldots,x_n]$.  In this section, we first
introduce the exact Jacobian SDP relaxation proposed in \cite{exactJacNie}.
Then we present our contribution in this section by giving a weakened
assumption under which the relaxation in \cite{exactJacNie} is still exact.

Let $m=\min\{m_1+m_2,n-1\}$.  For convenience, denote
$h(x)=(h_1(x),\ldots,h_{m_1}(x))$ and $g(x)=(g_1(x),\ldots,g_{m_2}(x))$. For a
subset $J=\{j_1,\ldots,j_k\}\subseteq [m_2]$, denote
$g_J(x)=(g_{j_1}(x),\ldots,g_{j_k}(x))$. Symbols $\nabla h(x)$ and $\nabla g_J(x)$
represent the gradient vectors of the polynomials in $h(x)$ and $g_J(x)$,
respectively. Denote the determinantal variety of
$(f,h,g_J)$'s Jacobian being singular by
\[
G_J=\left\{x\in\CC^n\mid\text{rank}\ B^J(x)\le m_1+|J|\right\},
\quad B^J(x)=\left[\nabla f(x)\quad\nabla h(x)\quad\nabla
g_J(x)\right].
\]
For every $J=\{j_1,\ldots,j_k\}\subseteq [m_2]$ with $k\le m-m_1$, let
$\eta^J_1,\ldots,\eta^J_{len(J)}$ be the set of defining polynomials for $G_J$
where $len(J)$ is the number of these polynomials.  See \cite[Section
2.1]{exactJacNie} about minimizing the number of defining equations for
determinantal varieties. For each $i=1,\ldots,len(J)$, define 
\begin{equation}\label{def::varphi}
\varphi_i^J(x)=\eta^J_i\cdot\underset{j\in J^c}{\prod}g_j(x),\ 
\text{where}\ J^c=[m_2]\backslash J.
\end{equation}
For simplicity, we list all possible $\varphi^J_i$ in (\ref{def::varphi})
sequentially as 
\begin{equation}\label{list::varphi}
\varphi_1,\varphi_2,\ldots,\varphi_r,\ \text{where}\ 
r=\underset{J\subseteq[m_2],|J|\le m-m_1}{\sum}len(J).
\end{equation}
Define the variety
\begin{equation}\label{def::W}
W:=\{x\in\CC^n\mid\ h_1(x)=\cdots=h_{m_1}(x)=\varphi_1(x)=\cdots=\varphi_r(x)=0\}.
\end{equation}
We consider the following optimization
\begin{equation}\label{pro::jacobian}
\left\{
\begin{aligned}
f^*:=\underset{x\in\RR^n}{\min}&\ f(x)\\
\text{s.t.}&\ h_i(x)=0\ (i=1,\ldots,m_1),\ \varphi_j(x)=0\ (j=1,\ldots,r),\\
&\ g_{\nu}(x)\ge 0, \forall \nu\in\{0,1\}^{m_2},
\end{aligned}\right.
\end{equation}
where $g_{\nu}=g_1^{\nu_1}\cdots g_{m_2}^{\nu_{m_2}}$.

We now construct $N$-th order SDP relaxation \cite{LasserreGlobal2001}
for (\ref{pro::jacobian}) and its dual problem. Let $\psi(x)$ be a
polynomial with $\deg{(\psi)}\le 2N$ and define symmetric matrices
$A^{(N)}_{\alpha}$ such that 
\[
\psi(x)[x]_d[x]_d^T=\underset{\alpha\in{\mathbb N}^n:|\alpha|\le 2N}{\sum}
A^{(N)}_{\alpha}x^{\alpha},\ \text{where}\ d=N-\lceil\deg{(\psi)}/2\rceil.
\]
Then the $N$-th order localizing moment matrix of $\psi$ is defined as
\begin{equation}\label{eq::momentmatrix}
L^{(N)}_{\psi}(y)=\underset{\alpha\in{\mathbb N}^n:|\alpha|\le 2N}{\sum}
A^{(N)}_{\alpha}y_{\alpha},
\end{equation}
where $y$ is a moment vector indexed by $\alpha\in{\mathbb N}^n$ with 
$|\alpha|\le 2N$. Denote 
\[
L_f(y)=\underset{\alpha\in{\mathbb N}^n:|\alpha|\le
\deg{(f)}}{\sum}f_{\alpha}y_{\alpha}\quad \text{for}\quad
f(x)=\underset{\alpha\in{\mathbb N}^n:|\alpha|\le
\deg{(f)}}{\sum}f_{\alpha}x^{\alpha}.
\]
The $N$-th order SDP relaxation \cite{LasserreGlobal2001} for
(\ref{pro::jacobian}) is the SDP
\begin{equation}\label{relax::primal}
\left\{
\begin{aligned}
f_N^{(1)}:=\min&\ L_f(y)\\
\text{s.t.}&\ L_{h_i}^{(N)}(y)=0\ (i=1,\ldots,m_1),\ L_{\varphi_j}^{(N)}(y)=0\
 (j=1,\ldots,r),\\
&\ L_{g_{\nu}}^{(N)}\succeq 0, \forall \nu\in \{0,1\}^{m_2}, y_0=1.
\end{aligned}\right.
\end{equation}
Now we present the dual of (\ref{relax::primal}). Define the truncated preordering 
$P^{(N)}$ generated by $g_j$ as 
\begin{equation*}
P^{(N)}=\left\{
\underset{\nu\in\{0,1\}^{m_2}}{\sum}\sigma_{\nu}g_{\nu}\Bigg|
\begin{aligned}
&\deg{(\sigma_{\nu}g_{\nu})}\le 2N\\
&\sigma_{\nu}\text{'s are SOS} 
\end{aligned}
\right\},
\end{equation*}
and the truncated ideal $I^{(N)}$ generated by $h_i$ and $\varphi_j$ as
\begin{equation*}
I^{(N)}=\left\{
\sum\limits_{i=1}^{m_1}\psi_ih_i+\sum\limits_{j=1}^r\phi_j\varphi_j\Bigg| 
\begin{aligned}
&\deg{(\psi_ih_i)}\le 2N\ \forall i\\
&\deg{(\phi_j\varphi_j)}\le 2N\ \forall j
\end{aligned}
\right\}.
\end{equation*}
It is shown \cite{LasserreGlobal2001} that the dual of (\ref{relax::primal}) 
is the following SOS relaxation for (\ref{pro::jacobian}):
\begin{equation}\label{relax::dual}
\left\{
\begin{aligned}
f^{(2)}_N:=\max&\ \gamma\\ 
\text{s.t.}&\ f(x)-\gamma\in I^{(N)}+P^{(N)}.
\end{aligned}\right.
\end{equation}
By weak duality, we have $f^{(2)}_N\le f^{(1)}_N\le f^*$.  
For any subset $J=\{j_1,\ldots,j_k\}\subseteq [m_2]$, let 
\[
V(h,g_J)=\{x\in\CC^n\mid h_i(x)=0,\ g_j(x)=0,\ i=1,\ldots,m_1,\ j\in J\}. 
\] 
We make the following assumption.

\begin{assumption}\label{assump::original}
\begin{inparaenum}[\upshape(i\upshape)]
\item\label{item::a1} $m_1\le n$. 
\item\label{item::a2} For any feasible point $u$, at most $n-m_1$ of $g_1(u),\ldots,g_{m_2}(u)$ vanish. 
\item\label{item::a3} For every $J=\{j_1,\ldots,j_k\}\subseteq [m_2]$ with $k\le n-m_1$,
Jacobian $[\nabla h\quad\nabla g_J]$ has full rank on $V(h,g_J)$.
\end{inparaenum}
\end{assumption}
Under the above assumption, the following main result is shown in \cite{exactJacNie}.
\begin{theorem}{\upshape\cite[Theorem 2.3]{exactJacNie}}\label{thm::main::Nie}
Suppose Assumption \ref{assump::original} holds. Then $f^*>-\infty$ and there 
exists $N^*\in{\mathbb N}$ such that $f^{(1)}_N=f^{(2)}_N=f^*$ for all $N\ge N^*$.
Furthermore, if the minimum $f_{min}$ of $(\ref{pro::opti})$ is achievable,
then $f^{(1)}_N=f^{(2)}_N=f_{min}$ for all $N\ge N^*$.
\end{theorem}
According to Theorem \ref{thm::main::Nie}, it is possible to solve the
polynomial optimization $(\ref{pro::opti})$ {\itshape exactly} by a single SDP
relaxation, which was not known in the prior existing literature. It is also
shown in {\upshape\cite{exactJacNie}} that Assumption \ref{assump::original} is
generically true. It is the reason why we use this method to solve
$(\ref{pro::homo})$. We will show later in this paper that Assumption
\ref{assump::original} is always true for $(\ref{pro::homo})$ when the original
feasible set $S=\RR^n$, i.e. for the global minimization of a rational
function.  In the following of this section, we prove that the condition
(\ref{item::a3}) in Assumption \ref{assump::original} can always be weakened
such that the conclusions in Theorem \ref{thm::main::Nie} still hold.

\begin{define}\label{def::theta}
For every set $J=\{j_1,\ldots,j_k\}\subseteq [m_2]$ with $k\le n-m_1$,
let
\[
\Theta_J=\{x\in V(h,g_J)\mid\text{rank}\ \left[\nabla h\quad\nabla
g_J\right]<m_1+|J|\}\quad\text{and}\quad\Theta=\bigcup\limits_{J\subseteq
[m_2], |J|\le n-m_1}\Theta_J.
\]
\end{define} 

We next show that the Jacobian SDP relaxation \cite{exactJacNie} is still exact
under the following weakened assumption:

\begin{assumption}\label{assump}
\begin{inparaenum}[\upshape(i\upshape)]
\item\label{item::ns1} $m_1\le n$. 
\item\label{item::ns2} For any $u\in S$, at most $n-m_1$ of
$g_1(u),\ldots,g_{m_2}(u)$ vanish. 
\item\label{item::ns3} The set $\Theta$ is finite.
\end{inparaenum}
\end{assumption}

Let $K$ be the variety defined by the KKT conditions
\begin{equation*}
      K=\left\{\begin{array}{c|c}
	(x,\lambda,\mu)\in\CC^{n+m_1+m_2}&
        \begin{aligned}
        & \nabla f(x)=\underset{i=1}{\overset{m_1}{\sum}}\lambda_i\nabla h_i(x)
	+\underset{j=1}{\overset{m_2}{\sum}}\mu_j\nabla g_j(x)\\
        &h_i(x)=\mu_jg_j(x)=0, \forall (i,j)\in [m_1]\times [m_2]
        \end{aligned}
    \end{array}
    \right\}
\end{equation*}
and 
\[
K_x=\{x\in\CC^n\mid (x,\lambda,\mu)\in K\ \text{for some}\ \lambda, \mu\}.
\]
Under  Assumption \ref{assump::original}, \cite[Lemma 3.1]{exactJacNie}
states that $W=K_x$. We now improve this result as follows.
\begin{lemma}[Revised Version of Lemma 3.1 in \cite{exactJacNie}]
\label{lem::WK}
Under conditions {\upshape$(\ref{item::ns1})$} and
{\upshape$(\ref{item::ns2})$} in Assumption \ref{assump}, $W=\Theta\cup K_x$.
\end{lemma}
\begin{proof}
The proof of \cite[Lemma 3.1]{exactJacNie} shows that
$W\backslash\Theta\subseteq K_x\subseteq W$.  With a similar argument, we prove
$\Theta\subseteq W$. Recall that $B^J=[\nabla f(x)\quad\nabla h(x)\quad\nabla
g_J(x)]$.  Choose an arbitrary $u\in\Theta$ and let $u\in\Theta_I$ for some
$I\subseteq [m_2]$. If $I=\emptyset$, then $[\nabla h]$ and $B^J(u)$ are both
singular for any $J\subseteq [m_2]$, which implies $\varphi_i(u)=0$ and $u\in
W$. If $I\neq\emptyset$, write $I=\{i_1,\ldots,i_t\}$. Let
$J=\{j_1,\ldots,j_k\}\subseteq [m_2]$ be an arbitrary index set with $m_1+k\le
m$.

\noindent $\mathbf{Case}$ $I\not\subseteq J$\quad At least one $j\in J^c$ belongs
to $I$. By the choice of $I$ and the definition of $\varphi_i(x)$,
\[
\varphi_i^J(u)=\eta^J_i\cdot\underset{j\in J^c}{\prod}g_j(u)=0. 
\]
\noindent $\mathbf{Case}$ $I\subseteq J$\quad Then $[\nabla h\quad\nabla g_I]$
and $[\nabla f(x)\quad\nabla h(x)\quad\nabla g_J(x)]$ are both singular. Hence,
all polynomials $\varphi_i^J(x)$'s vanish at $u$.

Combining the above two cases, we have all $\varphi_i^J(x)$ vanish at $u$.
Thus, $u\in W$ which implies $W=\Theta\cup K_x$.
\end{proof}
\begin{lemma}\label{lem::infimumeq}
Under conditions {\upshape(\ref{item::ns1})} and {\upshape(\ref{item::ns2})} in
Assumption \ref{assump}, if the minimum $f_{min}$ of
{\upshape(\ref{pro::opti})} is achievable, then $f^*=f_{min}$.
\end{lemma}
\begin{proof}
By the construction of (\ref{pro::jacobian}), $f^*\ge f_{min}$.  Suppose
$f_{min}=f(x^*)$ where $x^*$ is a feasible point of (\ref{pro::opti}).  If
$x^*\notin\Theta$, then the linear independence constraint qualification (LICQ)
is satisfied at $x^*$ which implies $x^*\in K_x$ \cite[Theorem
12.1]{NumericalOptimization}.  Since $W=\Theta\cup K_x$ by Lemma \ref{lem::WK},
we have $x^*\in W$ which implies $f^*=f_{min}$.
\end{proof}

Next we show that the conclusion in \cite[Lemma 3.2]{exactJacNie} still holds
under Assumption \ref{assump}.
\begin{lemma}[Revised Version of Lemma 3.2 in \cite{exactJacNie}]
Suppose Assumption \ref{assump} holds. Let $T=\{x\in\RR^n\mid g_j(x)\ge 0,\ 
j=1,\ldots,m_2\}$.  Then there exist disjoint subvarieties $W_0,W_1,\ldots,W_r$
of $W$ and distinct $v_1,\ldots,v_r\in\RR$ such that
\[
W=W_0\cup W_1\cup\cdots\cup W_r,\quad W_0\cap T=\emptyset,\quad W_i\cap T\neq
\emptyset,\quad i=1,\ldots,r,
\]
and $f(x)$ is constantly equal to $v_i$ on $W_i$ for $i=1,\ldots,r$.
\end{lemma}
\begin{proof} 
Denote $Zar(K_x)$ the Zariski closure of $K_x$ and let $\Omega=W\backslash
Zar(K_x)$.  By Lemma \ref{lem::WK}, we have $Zar(K_x)\subseteq W$ and
$\Omega\subseteq\Theta$.  With the proof of \cite[Lemma 3.2]{exactJacNie}, we
can conclude that there exist disjoint subvarieties $W_0,W_1,\ldots,W_t$ of
$Zar(K_x)$ and distinct $v_1,\ldots,v_t\in\RR$ such that
\[
Zar(K_x)=W_0\cup W_1\cup\cdots\cup W_t,\quad W_0\cap T=\emptyset,
\quad W_i\cap T\neq \emptyset,\quad i=1,\ldots,t,
\]
and $f(x)$ is constantly equal to $v_i$ on $W_i$ for $i=1,\ldots,t$.  We now
consider the set $\Omega$. Let $W_0=V(E_0)$, then for any
$u\in\Omega\cap\CC^n$, $W_0\cup\{u\}=V(E_0)\cup V(\langle x-u\rangle)
=V(\langle x-u\rangle\cdot E_0)$. Since $\Omega\cap\CC^n\subseteq\Theta$ is a
finite set by Assumption \ref{assump}, if we group $W_0$ and $\Omega\cap\CC^n$
together then we get a new subvariety.  We still denote it by $W_0$ for
convenience. Then $W_0\cap T=\emptyset$. Take any $w\in\Omega\cap\RR^n$, if
$f(w)=v_{i_0}$ for some $i_0\in\{1,\ldots,t\}$, then we put $w$ into $W_{i_0}$
and get a new subvariety by the same reason as $W_0$. We still write the
resulting subvariety as $W_{i_0}$. If for any $i\in\{1,\ldots,t\}$, $f(w)\neq
v_i$, then let $W_{t+1}=\{w\}$ and $v_{t+1}=f(w)\in\RR$. Since
$\Omega\cap\RR^n\subseteq\Theta$ is a finite set, the above process will
terminate and we can obtain the required decomposition of $W$.
\end{proof}

Since we get the same result as in \cite[Lemma 3.2]{exactJacNie} under the
weakened Assumption \ref{assump}, \cite[Theorem 3.4]{exactJacNie} which is
based on \cite[Lemma 3.2]{exactJacNie} can be restated as follows.

\begin{theorem}[Revised Version of Theorem 3.4 in \cite{exactJacNie}]\label{thm::exact}
Suppose Assumption \ref{assump} holds. Then $f^*>-\infty$ and there exists
$N^*\in{\mathbb N}$ such that for all $\varepsilon>0$
\begin{equation}\label{eq::exact}
f(x)-f^*+\varepsilon\in I^{(N^*)}+P^{(N^*)}.
\end{equation} 
\end{theorem}
Since $\varepsilon$ in (\ref{eq::exact}) is arbitrary, by Lemma \ref{lem::infimumeq},
Theorem \ref{thm::main::Nie} becomes
\begin{theorem}[Revised Version of Theorem 2.3 in \cite{exactJacNie}]\label{thm::main::exact}
Suppose Assumption \ref{assump} holds. Then $f^*>-\infty$ and there
exists $N^*\in{\mathbb N}$ such that $f^{(1)}_N=f^{(2)}_N=f^*$ for all $N\ge N^*$.
Furthermore, if the minimum $f_{min}$ of $(\ref{pro::opti})$ is achievable,
then $f^{(1)}_N=f^{(2)}_N=f_{min}$ for all $N\ge N^*$.
\end{theorem}
\begin{remark}
We now compare the conditions {\upshape(iii)} in Assumption
\ref{assump::original} and \ref{assump}. For any $J=\{j_1,\ldots,j_k\}\subseteq
[m_2]$ with $k\le n-m_1$, suppose the ideal $\langle h, g_J\rangle$ is radical
and its codimension is $m_1+|J|$. Then the condition {\upshape(iii)} in
Assumption \ref{assump::original} requires the variety $V(h,g_J)$ is nonsingular for
every subset $J$. In this section, we have proved that if the 
singularities of $V(h,g_J)$ are finite, i.e. the condition {\upshape(iii)} in
Assumption \ref{assump} holds, the Jacobian SDP relaxation
{\upshape\cite{exactJacNie}} is still exact. 
\end{remark}

\begin{cor}\label{cor::d=2}
Suppose that 
\begin{enumerate}[\upshape (a)]
\item\label{item::a} For each subset $J\subseteq [m_2]$ with $|J|\le n-m_1$,
$\langle h, g_J\rangle$ is a radical ideal and its codimension is $m_1+|J|$;  
\item\label{item::b} $V(h)$ is a smooth variety of dimension $\le 2$. 
\end{enumerate}
Then the condition {\upshape(\ref{item::ns3})} in Assumption \ref{assump}
always holds.  Therefore, if conditions {\upshape(\ref{item::ns1})} and
{\upshape(\ref{item::ns2})} in Assumption \ref{assump} are satisfied, then the
conclusions of Theorem \ref{thm::main::exact} hold.
\end{cor}
\begin{proof}
For any subset $J\subseteq [m_2]$ with $|J|\le n-m_1$, by (\ref{item::a}),
$\Theta_J$ is the set of singularities of $V(h,g_J)$. If $J=\emptyset$, then
$\Theta_J=\emptyset$ by (\ref{item::b}). If $J\neq\emptyset$, then by
\cite[Proposition 3.3.14]{realAG} or \cite[Theorem 5.3]{AG_Hartshorne}, $\dim
\Theta_J<\dim V(h,g_J)$. Since $\dim V(h,g_J)\le 1$ by (\ref{item::a}) and
(\ref{item::b}), $\Theta_J$ is a finite set for each $J\subseteq [m_2]$ with
$|J|\le n-m_1$.  Thus the condition (\ref{item::ns3}) in Assumption
\ref{assump} always holds. 
\end{proof}

We now give an example to illustrate the finite convergence of the Jacobian SDP
relaxation \cite{exactJacNie} under the weakened Assumption
\ref{assump}.
\begin{example}
{\upshape
Consider the following polynomial optimization
\begin{equation*}
\left\{
\begin{aligned}
\underset{x_1,x_2\in\RR}{\min}&\ f(x_1,x_2):=x_1x_2^2+x_1\\
\text{s.t.}&\ h(x_1,x_2):=-x_1^3+x_2^2=0.
\end{aligned}\right.
\end{equation*}
Clearly, the minimum $f_{min}=0$ is achieved at $(0,0)$. However, it is easy to
verify that $(0,0)$ is a singular point and does not satisfy the KKT
conditions. Since $(0,0)$ is the only singularity, Assumption \ref{assump}
holds which is also guaranteed by Corollary \ref{cor::d=2}.  In the following,
we show the finite convergence of the Jacobian SDP relaxation
{\upshape\cite{exactJacNie}} by giving the exact equation $(\ref{eq::exact})$.  

By the construction of $(\ref{pro::jacobian})$, $m_1=1,m_2=0$ and $r=1$. 
$\varphi(x_1,x_2):=2x_2(x_2^2+1)+6x_1^3x_2$. For any $\varepsilon>0$, let
\begin{equation*}
\begin{aligned}
\psi(x_1,x_2):=&8x_1+8\varepsilon-12x_1^8x_2^4-24x_1^8x_2^2+24x_1^7x_2^2+8x_1x_2^2+4x_1^3+32\varepsilon
x_1^3-\frac{x_1^3}{\varepsilon^2}\\
&+\frac{x_1^4}{8\varepsilon^3}-\frac{2x_1^6}{\varepsilon^2}+\frac{x_1^7}{4\varepsilon^3} +
8\varepsilon x_2^2+\frac{x_1x_2^2}{64\varepsilon^3}-
\frac{x_2^2}{8\varepsilon^2}+\frac{x_1}{64\varepsilon^3}-\frac{1}{8\varepsilon^2}+
4x_1^3x_2^2\\
&+4x_1^5x_2^2+4x_1^5+24x_1^4-\frac{243x_1^{10}x_2^2}{256\varepsilon^3}-\frac{3x_1^{10}x_2^6}{16\varepsilon^3}-
\frac{45x_1^7x_2^4}{128\varepsilon^3}-\frac{311x_1^7x_2^2}{1024\varepsilon^3}\\
&-\frac{3x_1^7x_2^6}{32\varepsilon^3}+\frac{3x_1^3x_2^4}{32\varepsilon^2}+
\frac{33x_1^9x_2^2}{8\varepsilon^2}+\frac{29x_1^6x_2^2}{32\varepsilon^2}-
\frac{45x_1^4x_2^4}{1024\varepsilon^3}-\frac{45x_1^{10}x_2^4}{64\varepsilon^3}\\
&-\frac{17x_1^3x_2^2}{32\varepsilon^2}+\frac{3x_1^9x_2^4}{2\varepsilon^2}+
\frac{3x_1^6x_2^4}{4\varepsilon^2}+\frac{47x_1^4x_2^2}{1024\varepsilon^3}-
\frac{3x_1^4x_2^6}{256\varepsilon^3}.\\
\end{aligned}
\end{equation*}
\begin{equation*}
\begin{aligned}
\phi(x_1,x_2):=&-\frac{x_1^{10}x_2^5}{32\varepsilon^3}-\frac{15x_1^{10}x_2^3}{128\varepsilon^3}
+\frac{x_1^9x_2^3}{4\varepsilon^2}-\frac{x_1^7x_2^5}{64\varepsilon^3}-
\frac{81x_1^{10}x_2}{512\varepsilon^3}-2x_1^8x_2^3+\frac{11x_1^9x_2}{16\varepsilon^2}\\
&-\frac{15x_1^7x_2^3}{256\varepsilon^3}-4x_1^8x_2+\frac{x_1^6x_2^3}{8\varepsilon^2}
-\frac{x_1^4x_2^5}{512\varepsilon^3}-\frac{337x_1^7x_2}{2048\varepsilon^3}+
4x_1^7x_2+\frac{59x_1^6x_2}{64\varepsilon^2}\\
&-\frac{15x_1^4x_2^3}{2048\varepsilon^3}-2x_1^5x_2+
\frac{x_1^3x_2^3}{64\varepsilon^2}-\frac{x_1^4x_2}{16\varepsilon^3}+
\frac{7x_1^3x_2}{16\varepsilon^2}-2x_1^3x_2-4x_1x_2\\
&-\frac{x_1x_2}{128\varepsilon^3}+\frac{x_2}{16\varepsilon^2}-4\varepsilon x_2.
\end{aligned}
\end{equation*}
\begin{equation*}
\begin{aligned}
\sigma_0(x_1,x_2):=&16\left(\varepsilon+\frac{(x_1x_2^2+x_1+1)^2}{4}+(x_1x_2^2+x_1-1)^2x_2^2\right)x_1^6+
\varepsilon(4x_1^3+1)^2\\&\left(1+\frac{x_1x_2^2+x_1}{2\varepsilon}-\frac{(x_1x_2^2+x_1)^2}{8\varepsilon^2}\right)^2.\\
\end{aligned}
\end{equation*}
It can be verified that 
\[
f(x_1,x_2)+\varepsilon = \sigma_0(x_1,x_2) + \psi(x_1,x_2)h(x_1,x_2) + \phi(x_1,x_2)\varphi(x_1,x_2).
\]
Since each term on the right side of the above equation has degree $\le 20$, 
we take $N^*=10$ in $(\ref{eq::exact})$. Because $\sigma_0(x_1,x_2)$ is a sum of
squares of polynomials, we have $\sigma_0(x_1,x_2)\in P^{(10)}$ and
$\psi(x_1,x_2)h(x_1,x_2) + \phi(x_1,x_2)\varphi(x_1,x_2)\in I^{(10)}$.  Therefore,
$f(x_1,x_2)+\varepsilon\in I^{(10)}+P^{(10)}$ for any $\varepsilon>0$.  Hence, we
have $f^{(1)}_{N}=f^{(2)}_N=f_{min}=0$ for all $N\ge 10$.
}\hfill$\square$
\end{example} 

A practical issue in applications is how to detect whether
(\ref{relax::primal}) is exact for a given $N$. Nie \cite{exactJacNie} pointed
out that it would be possible to apply the {\itshape flat-extension condition}
(FEC) \cite{LasserreHenrion}. When FEC holds, (\ref{relax::primal}) is exact
for (\ref{pro::opti}) and a very nice software {\itshape GloptiPoly}
\cite{gloptipoly} provides routines for finding minimizers if FEC holds.  In
general, the FEC is a sufficient but not necessary condition for checking
finite convergence of Lasserre's hierarchy. More recently, Nie
\cite{NieFlatTruncation} proposed the {\itshape flat truncation} as a general
certificate. For the polynomial optimization (\ref{pro::opti}), define 
\begin{equation*}
\begin{aligned}
&d_{h,i}=\lceil\deg(h_i)/2\rceil,\quad d_{g,j}=\lceil\deg(g_j)/2\rceil, 
\quad d_f=\lceil\deg(f)/2\rceil,\\
&\hat{d}=\max\{1, d_{h,1},\ldots,d_{h,m_1}, d_{g,1},\ldots, d_{g,m_2}\}.
\end{aligned}
\end{equation*}
When $\psi\equiv 1$, $L_{\psi}^{(N)}(y)$ in (\ref{eq::momentmatrix}) is called
{\itshape moment matrix} and is denoted as $M_N(y):=L_{\psi}^{(N)}(y)$. For a
given integer $N\in{\mathbb N}$, we say an optimizer $y^*$ of
(\ref{relax::primal}) has a flat truncation if there exists an integer
$t\in[\max\{d_f,\hat{d}\},N]$ such that 
\[
\text{rank}\ M_{t-\hat{d}}(y^*)\quad=\quad\text{rank}\ M_t(y^*).
\]
Assuming the set of global minimizers is nonempty and finite, \cite[Theorem 2.2
and 2.6]{NieFlatTruncation} show that the Putinar type or Schm{\"u}dgen type
Lasserre's hierarchy has finite convergence if and only if the flat truncation
holds. As an application, \cite[Corollary 4.2]{NieFlatTruncation} also points
out that if (\ref{pro::opti}) has a nonempty set of finitely many global
minimizers and Assumption \ref{assump::original} is satisfied, then the flat
truncation is always satisfied for the hierarchy of Jacobian SDP relaxations.
Since we have proved that Assumption \ref{assump::original} can be weakened as
Assumption \ref{assump}, we have 

\begin{cor}[Revised Version of Corollary 4.2 in \cite{NieFlatTruncation}]
Suppose {\upshape(\ref{pro::opti})} has a nonempty set of finitely many global
minimizers and Assumption \ref{assump} is satisfied. Then, for all $N$ big
enough, the optimal value of {\upshape(\ref{relax::dual})} equals the global
minimum of {\upshape(\ref{pro::opti})} and every minimizer of
{\upshape(\ref{relax::primal})} has a flat truncation.
\end{cor}

\section{Revisiting Minimization of Rational Functions}\label{sec::rf2}

In this section, we return to the minimization of (\ref{pro::rf}). We first
apply the Jacobian SDP relaxation discussed in Section \ref{sec::exactrelax} to
reformulate (\ref{pro::homo}) as (\ref{pro::jacobian}) for which we consider 
the finite convergence of the SDP relaxations. Next, we do some
numerical experiments to show the efficiency of our method.

\subsection{Minimizing Rational Functions by  Jacobian SDP Relaxation}

Consider the number of new constraints added when we employ Jacobian SDP
relaxation to solve (\ref{pro::homo}). As mentioned in \cite{exactJacNie}, the
number of new constraints in (\ref{pro::jacobian}) is exponential in the number
of inequality constraints.  Hence, if the number of inequality constraints is
large, (\ref{pro::jacobian}) becomes more difficult to solve numerically. In
the following, we employ the Jacobian SDP relaxation to reformulate
(\ref{pro::homo}) as (\ref{pro::jacobian}). We show that the number of the new
equality constraints $\varphi_i$'s in (\ref{pro::jacobian}) can be reduced due
to the special inequality constraint $x_0\ge 0$ in (\ref{pro::homo}). 

In (\ref{pro::homo}), for convenience, let 
\[
h^{hom}_{m_1+1}(\tilde{x}):=\tilde{q}(\tilde{x})-1=0,\quad 
g_{m_2+1}^{hom}(\tilde{x}):=x_0\ge 0\quad\text{and}\quad m:=\min\{m_1+m_2+2,n\}. 
\]
Denote 
\[
\nabla_{\tilde{x}}:=\left(\frac{\partial}{\partial
x_0},\frac{\partial}{\partial x_1}, \cdots,\frac{\partial}{\partial
x_n}\right).
\]
According to (\ref{def::varphi}) and (\ref{list::varphi}), we need to consider
all subsets of $[m_2+1]$ with cardinality $\le m-m_1-1$. Let
$l=\min\{m-m_1-1,m_2\}$.  We first consider the subsets without $m_2+1$, i.e.,
every subset $J=\{j_1,\ldots,j_k\}\subseteq [m_2]$ with $k\le l$.  Denote
$h^{hom}=(h_1^{hom},\ldots,h^{hom}_{m_1},h^{hom}_{m_1+1})$ and
$g^{hom}_{J}=(g^{hom}_{j_1},\ldots,g^{hom}_{j_k})$.  Let
$\{\eta_1,\ldots,\eta_{len(J)}\}$ be the set of the defining equations for the
determinantal variety 
\[
G_{J}:=\{\tilde{x}\in\CC^{n+1}\mid \text{rank}\
[\nabla_{\tilde{x}}\tilde{p}\quad \nabla_{\tilde{x}}h^{hom}\quad
\nabla_{\tilde{x}} g^{hom}_{J}]\le m_1+|J|+1\}.
\]
For each $i=1,\cdots,len(J)$, define 
\[
\varphi_i^J(\tilde{x})=\eta_i\cdot\underset{j\in
J^c}{\prod}g_j^{hom}(\tilde{x}), \quad\text{where}\quad J^c=[m_2+1]\backslash
J.
\]
For every subset $J$ considered above, denote $J'=J\cup\{m_2+1\}\subseteq
[m_2+1]$. It can be checked that the collection of these $J$'s and $J'$'s
contains all subsets of $[m_2+1]$ with cardinality $\le m-m_1-1$ and some
possible $J'$'s with cardinality $= m-m_1$ (which will happen when
$n<m_1+m_2+2$).

\noindent $\mathbf{Case}$ $|J'|\le m-m_1-1$\quad 
All these $J'$'s compose of the subsets of $[m_2+1]$ containing $m_2+1$ with
cardinality $\le m-m_1-1$.  It is easy to see that the set of the defining
equations for the determinantal variety
\[
G_{J'}:=\{\tilde{x}\in\CC^{n+1}\mid \text{rank}\
[\nabla_{\tilde{x}}\tilde{p}\quad \nabla_{\tilde{x}}h^{hom}\quad
\nabla_{\tilde{x}} g^{hom}_{J}\quad \nabla_{\tilde{x}} x_0]\le m_1+|J|+2\}
\]
is a subset of $\{\eta_1,\ldots,\eta_{len(J)}\}$.  We generally suppose it to
be $\{\eta_1,\ldots,\eta_{t(J)}\}$ with $t(J)<len(J)$.  For $i=1,\cdots,t(J)$,
define
\[
\varphi_i^{J'}(\tilde{x})=\eta_i\cdot\underset{j\in
{J'}^c}{\prod}g_j^{hom}(\tilde{x}), \quad\text{where}\quad {J'}^c=[m_2+1]\backslash
J'.
\]

\noindent $\mathbf{Case}$ $|J'|=m-m_1$\quad 
It is easy to check that $G_{J'}=\CC^{n+1}$. Thus for convenience, we set
$t(J)=0$ in this case. 

Then for every subset $J\subseteq [m_2]$ with $|J|\le l$, we have 
\begin{equation}\label{eq::phi}
\varphi_i^J(\tilde{x})=\varphi_i^{J'}(\tilde{x})\cdot x_0,\quad
i=1,\cdots,t(J). 
\end{equation}
Now consider the SDP relaxations \cite{LasserreGlobal2001} for the following
polynomial optimization
\begin{equation}\label{pro::rfjac}
\left\{
\begin{aligned}
p^*:=\underset{\tilde{x}\in\RR^{n+1}}{\min}&\ \tilde{p}(\tilde{x})\\
\text{s.t.}&\ h^{hom}_1(\tilde{x})=\cdots=h^{hom}_{m_1}(\tilde{x})=
h^{hom}_{m_1+1}(\tilde{x})=0,\\
&\ \varphi_i^J(\tilde{x})=0, \varphi_j^{J'}(\tilde{x})=0 \\ 
&\ (i\in[len(J)], j\in[t(J)], J\subseteq
[m_2], |J|\le l),\\ &\ g^{hom}_{\nu}(\tilde{x})\ge 0, \forall
\nu\in\{0,1\}^{m_2+1},
\end{aligned}\right.
\end{equation}
where $g^{hom}_{\nu}=(g^{hom}_1)^{\nu_1}\cdots
(g^{hom}_{m_2+1})^{\nu_{m_2+1}}$.  We now show that for each $J\subseteq [m_2]$
with $|J|\le l$, constraints
$\varphi_1^J(\tilde{x})=\cdots=\varphi_{t(J)}^J(\tilde{x})=0$ can be removed
from (\ref{pro::rfjac}). Consider the $N$-th order SDP relaxation
(\ref{relax::primal}) for (\ref{pro::rfjac}).  By (\ref{eq::phi}) and the
properties of localizing moment matrices in \cite[Lemma 4.1]{Laurent_sumsof},
we have 
\[
L_{\varphi_j^{J'}}^{(N)}(y)=0\quad \text{implies}\quad
L_{\varphi_j^{J}}^{(N)}(y)=0,\quad j=1,\ldots,t(J),\ J\subseteq [m_2],\ |J|\le
l.
\]
In the dual problem (\ref{relax::dual}), by (\ref{eq::phi}), the truncated
ideal 
\begin{equation*}
\left\{
\sum\limits_{J\subseteq [m_2], |J|\le l}
\left(\sum\limits_{i=1}^{len(J)}\phi_i\varphi_i^J+
\sum\limits_{j=1}^{t(J)}\zeta_j\varphi_j^{J'}\right)
+\sum\limits_{k=1}^{m_1+1}\psi_kh^{hom}_k
\right\},\ \text{where}\
\end{equation*}
\[
\forall i,j,k,\ 
\deg{(\phi_i\varphi_i^J)}\le 2N, \deg{(\zeta_j\varphi_j^{J'})}\le 2N, 
\deg{(\psi_kh^{hom}_k)}\le 2N,
\]
agrees with 
\begin{equation}\label{def::I}
\left\{
\sum\limits_{J\subseteq [m_2], |J|\le l}
\left(\sum\limits_{i=t(J)+1}^{len(J)}\phi_i\varphi_i^J+
\sum\limits_{j=1}^{t(J)}\zeta_j\varphi_j^{J'}\right)
+\sum\limits_{k=1}^{m_1+1}\psi_kh^{hom}_k
\right\}\ \text{where}
\end{equation}
\[
\forall i,j,k,\ \deg{(\phi_i\varphi_i^J)}\le 2N, 
\deg{(\zeta_j\varphi_j^{J'})}\le 2N, 
\deg{(\psi_kh^{hom}_k)}\le 2N.
\]
Therefore, we can remove
$\varphi_1^J(\tilde{x})=\cdots=\varphi_{t(J)}^J(\tilde{x})=0$ in
(\ref{pro::rfjac}) and improve the numerical performance in practice. Hence we
consider the following optimization 
\begin{equation}\label{pro::rfjacReduced}
\left\{
\begin{aligned}
p^*:=\underset{\tilde{x}\in\RR^{n+1}}{\min}&\ \tilde{p}(\tilde{x})\\
\text{s.t.}&\ h^{hom}_1(\tilde{x})=\cdots=h^{hom}_{m_1}(\tilde{x})
=h^{hom}_{m_1+1}(\tilde{x})=0,\\ 
&\ \varphi_i^J(\tilde{x})=0, \varphi_j^{J'}(\tilde{x})=0 \\
&\ (i=t(J)+1,\ldots,len(J), j\in[t(J)], J\subseteq [m_2], |J|\le l),\\
&\ g^{hom}_{\nu}(\tilde{x})\ge 0, \forall \nu\in\{0,1\}^{m_2+1},
\end{aligned}\right.
\end{equation}
The $N$-th order SDP relaxation \cite{LasserreGlobal2001} for
(\ref{pro::rfjacReduced}) is the SDP
\begin{equation}\label{relax::rfprimal}
\left\{
\begin{aligned}
p_N^{(1)}:=\min&\ L_{\tilde{p}}(y)\\
\text{s.t.}&\ L_{h^{hom}_1}^{(N)}(y)=\cdots=L_{h^{hom}_{m_1}}^{(N)}(y)=
L_{h^{hom}_{m_1+1}}^{(N)}(y)=0,\\
&\ L_{\varphi_i^J}^{(N)}(y)=0, L_{\varphi_j^{J'}}^{(N)}(y)=0 \\
&\ (i=t(J)+1,\ldots,len(J), j\in[t(J)], J\subseteq [m_2], |J|\le l),\\
&\ L_{g^{hom}_{\nu}}^{(N)}\succeq 0, \forall \nu\in \{0,1\}^{m_2+1}, y_0=1.
\end{aligned}\right.
\end{equation}
The dual problem of (\ref{relax::rfprimal}) is 
\begin{equation}\label{relax::rfdual}
\begin{aligned}
p^{(2)}_N:=\underset{\gamma\in\RR^{n+1}}{\max}\ \gamma\quad 
\text{s.t.}\ \tilde{p}(\tilde{x})-\gamma\in I^{(N)}+P^{(N)}.
\end{aligned}
\end{equation}
where $I^{(N)}$ is the ideal defined in (\ref{def::I}) and 
\begin{equation*}
P^{(N)}=\left\{
\underset{\nu\in\{0,1\}^{m_2+1}}{\sum}\sigma_{\nu}g^{hom}_{\nu}\Bigg|
\begin{aligned}
&\deg{(\sigma_{\nu}g^{hom}_{\nu})}\le 2N\\
&\sigma_{\nu}\text{'s are SOS} 
\end{aligned}
\right\}.
\end{equation*}

\begin{define}
For every set $J=\{j_1,\ldots,j_k\}\subseteq [m_2+1]$ with $k\le n-m_1$,
let
\[
\Theta_J=\{\tilde{x}\in V(h^{hom},g^{hom}_J)\mid\text{rank}\
\left[\nabla_{\tilde{x}} h^{hom}\quad\nabla_{\tilde{x}}
g^{hom}_J\right]<m_1+|J|+1\}
\]
and 
\[\Theta=\bigcup\limits_{J\subseteq [m_2+1],\ |J|\le
n-m_1}\Theta_J.
\] 
\end{define} 

\begin{assumption}\label{assumprf}
\begin{inparaenum}[\upshape(i\upshape)]
\item $m_1\le n$;
\item For any $u\in \widetilde{S}$ in {\upshape(\ref{def::S})}, 
at most $n-m_1$ of $g^{hom}_1(u),\ldots,g^{hom}_{m_2+1}(u)$ vanish;
\item The set $\Theta$ is finite.
\end{inparaenum}
\end{assumption}
By Theorem \ref{thm::main} and \ref{thm::main::exact}, we have 
\begin{theorem}\label{thm::rf}
Suppose Assumption \ref{assumprf} holds. Then $p^*>-\infty$ in
{\upshape(\ref{pro::rfjacReduced})} and there exists $N^*\in{\mathbb N}$ such
that $p^{(1)}_N=p^{(2)}_N=p^*$ for all $N\ge N^*$.  Furthermore, if one of the
conditions in Theorem \ref{thm::main} holds and the minimum $s^*$ of
{\upshape(\ref{pro::homo})} is achievable, then $p^{(1)}_N=p^{(2)}_N=r^*$ for
all $N\ge N^*$.
\end{theorem}
\begin{cor}
If $S=\RR^n$ in {\upshape(\ref{pro::rf})} and $s^*$ is achievable in
{\upshape(\ref{pro::homo::global})}, then there exists $N^*\in{\mathbb N}$ such
that $p^{(1)}_N=p^{(2)}_N=r^*$ for all $N\ge N^*$ in
{\upshape(\ref{relax::rfprimal})} and {\upshape(\ref{relax::rfdual})}.
\end{cor}
\begin{proof}
Since $\tilde{q}$ is homogeneous, regarding $\nabla\tilde{q}$ and $\tilde{x}$
as vectors in $\RR^{n+1}$, then $d\cdot\tilde{q}=\nabla\tilde{q}^T\cdot
\tilde{x}$ by Euler's Formula. Thus $\nabla(\tilde{q}-1)=\nabla\tilde{q}=0$
implies $\tilde{q}=0$, i.e. $\Theta=\emptyset$. Hence, Assumption \ref{assumprf} 
is always true for $(\ref{pro::homo::global})$. Then by Corollary \ref{cor::eq} 
and Theorem \ref{thm::rf}, the conclusion follows.
\end{proof}
In the end of this subsection, we would like to point out that $s^*$ in
(\ref{pro::homo}) might not be achievable in some cases. If the infimum of a
constrained polynomial optimization is asymptotic value, some approaches are
proposed in \cite{GGSZ2011,Vui-constraints}. Hence, we can still use these
approaches to solve (\ref{pro::homo}). However, to the best knowledge of the
authors, the finite convergence for these methods is unknown.

\subsection{Numerical Experiments}\label{sec::ex}
In this subsection, we present some numerical examples to illustrate the
efficiency of our method for solving minimization of (\ref{pro::rf}).  We use
the software {\itshape GloptiPoly} \cite{gloptipoly} to solve
(\ref{relax::rfprimal}) and (\ref{relax::rfdual}).
\subsubsection{Unconstrained rational optimization}
In the following, Example \ref{ex::m1} and \ref{ex::commonzeros} are
constructed from the {\itshape Motzkin} polynomial
\begin{equation}\label{eq::motzkin}
M(x_1,x_2,x_3)=x_1^4x_2^2+x_1^2x_2^4+x_3^6-3x_1^2x_2^2x_3^2.
\end{equation} 
As is well-known, $M(x_1,x_2,x_3)$ is nonnegative on $\RR^3$ but not SOS 
\cite{Reznick96someconcrete}.
\begin{example}\cite[Example 2.9]{Nie2008}\label{ex::m1}
{\upshape
Consider minimization
\begin{equation}\label{ex::I}
\min\limits_{x_1,x_2\in\RR} r(x_1,x_2):=\frac{x_1^4x_2^2+x_1^2x_2^4+1}{x_1^2x_2^2}.
\end{equation}
Taking $x_3=1$ in Motzkin polynomial, we have
$x_1^4x_2^2+x_1^2x_2^4+1-3x_1^2x_2^2\ge 0$ on $\RR^2$. Since $r(1,1)=3$, we
have $r^*=3$ and there are four minimizers $(\pm 1,\pm 1)$.  However,
$x_1^4x_2^2+x_1^2x_2^4+1-r^* x_1^2x_2^2$ is not SOS.  To solve this problem,
the authors in \cite{Nie2008} used the generalized big ball technique.  More
specifically, it is assumed that one of the minimizers of {\upshape(\ref{ex::I})} lies in
a ball $B(c,\rho)$ and the numerator and denominator of $r(x_1,x_2)$ have no
common real roots on $B(c,\rho)$. However, it is not easy in general to
determine the radius $\rho$ of this ball. We now solve this problem using our
method without the assumptions in \cite{Nie2008}.

We first
reformulate the problem as the following polynomial optimization:
\begin{equation*}
\left\{
 \begin{aligned}
\underset{x_0,x_1,x_2\in\RR}{\min}&\ \tilde{p}(x_0,x_1,x_2): = x_1^4x_2^2+x_1^2x_2^4+x_0^6\\
\text{s.t.}&\ \tilde{q}(x_0,x_1,x_2) := x_1^2x_2^2x_0^2=1.
\end{aligned}\right.
\end{equation*}
By Jacobian SDP relaxation (\ref{pro::rfjacReduced}), we need 3 more equations:
\begin{equation*}
\begin{aligned}
&\varphi_1(x_0,x_1,x_2) = 4x_1^3x_2^3x_0^2(x_1^2-x_2^2)=0\\
&\varphi_2(x_0,x_1,x_2) =4x_1x_2^2x_0(2x_1^4x_2^2+x_1^2x_2^4-3x_0^6)=0\\
&\varphi_3(x_0,x_1,x_2) =4x_1^2x_2x_0(x_1^4x_2^2+2x_1^2x_2^4-3x_0^6)=0
\end{aligned}
\end{equation*}
By the condition $x_1^2x_2^2x_0^2=1$, the above three
equations can be simplified as
\begin{equation*}
\begin{aligned}
&\varphi_1(x_0,x_1,x_2) = x_1^2-x_2^2=0\\
&\varphi_2(x_0,x_1,x_2) = 2x_1^4x_2^2+x_1^2x_2^4-3x_0^6=0\\
&\varphi_3(x_0,x_1,x_2) = x_1^4x_2^2+2x_1^2x_2^4-3x_0^6=0
\end{aligned}
\end{equation*}
We need to solve the following new problem
\begin{equation*}
\left\{
 \begin{aligned}
\underset{x_0,x_1,x_2\in\RR}{\min}&\ x_1^4x_2^2+x_1^2x_2^4+x_0^6\\
\text{s.t.}&\ x_1^2x_2^2x_0^2-1=0,\ 2x_1^4x_2^2+x_1^2x_2^4-3x_0^6=0,\\
&\ x_1^2-x_2^2=0,\ x_1^4x_2^2+2x_1^2x_2^4-3x_0^6 =0.
\end{aligned}\right.
\end{equation*}
Using GloptiPoly to solve this problem, we get the following results:
\begin{itemize}
\item $N=3$. The optimum is 3, but extracting global optimal solutions fails.
\item $N=4$. We get 8 optimal solutions for $(x_0,x_1,x_2)$: $(\pm 1,\pm 1,\pm
1)$ from which we get all the optimal solutions for original
problem: $(\pm1,\pm1)$. \hfill $\square$
\end{itemize}}  
\end{example}

\begin{example}\cite[Example 2.10]{Nie2008}\label{ex::commonzeros}
{\upshape
Consider the following problem
\begin{equation}\label{ex::1}
\underset{x_1,x_2\in\RR}{\min}\ r(x_1,x_2) := \frac{p(x_1,x_2)}{q(x_1,x_2)}
=\frac{x_1^4+x_1^2+x_2^6}{x_1^2x_2^2}.
\end{equation}
Taking $x_2=1$ in Motzkin polynomial {\upshape(\ref{eq::motzkin})}, we have $r^*=3$ with
$4$ minimizers $(\pm 1,\pm 1)$.  The denominator and numerator have real common
root $(0,0)$. In {\upshape\cite{Nie2008}}, the SOS relaxation
extracts $6$ solutions, $2$ of which are not global minimizers but the common
roots of $p(x)$ and $q(x)$. We reformulate it as the following polynomial
optimization and solve it by Jacobian SDP relaxation.
\begin{equation}
\left\{
 \begin{aligned}\label{ex::2}
\underset{x_0,x_1,x_2\in\RR}{\min}&\ \tilde{p}(x_0,x_1,x_2):= x_1^4x_0^2+x_1^2x_0^4+x_2^6\\
\text{s.t.}&\ \tilde{q}(x_0,x_1,x_2) :=   x_1^2x_2^2x_0^2=1.
\end{aligned}\right.
\end{equation}
Using GloptiPoly, we can still extract 8 solutions of {\upshape(\ref{ex::2})} and obtain
all the 4 optimal solutions of {\upshape(\ref{ex::1})} as in Example \ref{ex::m1}. In
our method, the constraint $\tilde{q}(x_0,x_1,x_2)=1$ prevents extracting the
common real roots of $p(x)$ and $q(x)$. This example also shows that Condition
{\upshape(\ref{item::2})} in Theorem \ref{thm::main2} is only sufficient but not
necessary. }\hfill $\square$ 
\end{example}

The following example is generated from the {\itshape Robinson} polynomial
\[
x^6_1+x^6_2+x_3^6-(x^4_1x_2^2+x_1^2x_2^4+x_1^4x_3^2+x_1^2x_3^4+x_2^4x_3^2+x_2^2x_3^4)+3x^2_1x^2_2x_3^2
\]
which is nonnegative on $\RR^3$ but not SOS \cite{Reznick96someconcrete}.
\begin{example}\label{ex::robinson}
{\upshape
Consider the following problem
\begin{equation}\label{exm:rob:0}
\begin{aligned}
\min\limits_{x_1,x_2\in\RR}&\ r(x_1,x_2):=\frac{p(x_1,x_2)}{q(x_1,x_2)} =
\frac{x^6_1+x^6_2+3x^2_1x^2_2+1}{x^2_1(x^4_2+1)+x^2_2(x^4_1+1)+(x^4_1+x^4_2)}.\\
\end{aligned}
\end{equation}

Taking $x_3=1$ in Robinson polynomial, we have $p(x_1,x_2 )-q(x_1,x_2)\geq 0$
on $\RR^2$. Since $r(1,1)=1$, $r^*=1$. We reformulate it as the following 
optimization: 
\begin{equation}
\left\{
\begin{aligned}\label{exm:rob:1}
\min\limits_{x_0,x_1,x_2\in\RR}&\ x^6_1+x^6_2+3x^2_1x^2_2x^2_0+x^6_0 \\
\text{s.t.}&\ x^2_1(x^4_2+x^4_0)+x^2_2(x^4_1+x^4_0)+x^2_0(x^4_1+x^4_2)=1.\\
\end{aligned}\right.
\end{equation}
The numerical results we obtained are:
\begin{itemize}
\item For relaxation order $N=5,6$, 
we get the optimum $s^*=1$, but the minimizers can not be extracted. 
\item For relaxation order $N=7$, we  
extract 20 approximate minimizers of {\upshape(\ref{exm:rob:1})}: 
\begin{equation*}
\begin{aligned}
&(-0.0000,\pm 0.8909,\pm 0.8909),&\quad & (\pm 0.8909,\pm 0.8909,-0.0000),\\
&(\pm 0.8909,-0.0000, \pm 0.8909),&\quad & (\pm 0.7418,\pm 0.7418,\pm 0.7418).
\end{aligned}
\end{equation*}
\end{itemize}
The above solutions correspond to the exact minimizers of (\ref{exm:rob:1}): 
\begin{equation*}
\begin{aligned}
&(0,\pm\frac{1}{\sqrt[6]{2}},\pm\frac{1}{\sqrt[6]{2}}),\quad 
(\pm\frac{1}{\sqrt[6]{2}},\pm\frac{1}{\sqrt[6]{2}},0),\\ 
&(\pm\frac{1}{\sqrt[6]{2}},0,\pm\frac{1}{\sqrt[6]{2}}),\quad 
(\pm\frac{1}{\sqrt[6]{6}},\pm\frac{1}{\sqrt[6]{6}}, \pm\frac{1}{\sqrt[6]{6}}).
\end{aligned}
\end{equation*}
There are four solutions with the first coordinate $x^*_0=0$ which indicate
that minimum $r^*=1$ is also an asymptotic value at $\infty$ by Theorem
\ref{thm::main2}. In fact, 
\begin{equation*}
\lim\limits_{x_1,x_2\rightarrow\infty}\frac{p(x_1,x_2)}{q(x_1,x_2)}=1=r^*.
\end{equation*} 
From the other 16 solutions, according to {\upshape(\ref{item::1})} in Theorem
\ref{thm::main2}, we get $8$ global minimizers of {\upshape(\ref{exm:rob:0})}:
$(\pm1,\pm1)$, $(\pm1,0)$, $(0,\pm1)$. \hfill $\square$
}
\end{example}

\begin{example}\cite[Example 3.4]{Nie2008}
{\upshape
Suppose function $\psi(z)$ and $\phi(z)$ are monic complex univariate
polynomials of degree $m$ such that:
$$\psi(z) =z^m+\psi_{m-1}z^{m-1}+\cdots +\psi_1z+\psi_0$$
$$\phi(z) =z^m+\phi_{m-1}z^{m-1}+\cdots +\phi_1z+\phi_0$$
It is shown in \cite{KarmarkarLakshmanGCDS} that finding nearest GCDs
becomes the following global minimization of rational functions 
\begin{equation}\label{gcd:exm}
\min\limits_{x_1,x_2\in\mathbb{R}}\ 
\frac{p(x_1,x_2)}{q(x_1,x_2)}=\frac{|\psi(x_1+ix_2)|^2+|\phi(x_1+ix_2)|^2}
{\sum\limits_{k=0}^{m-1}(x^2_1+x^2_2)^k}
\end{equation}
where $\deg(p) = 2m$ and $\deg(q) = 2(m-1)$. Let 
\[
\psi(z) = z^3+z^2-2,\quad \phi(z)= z^3+1.5z^2+1.5z-1.25.
\] 
Using our method for relaxation order $N=5$, we get four optimal solutions of
the  optimization reformulated from {\upshape(\ref{gcd:exm})} by
homogenization: 
\[
(0.7050,-0.7073,\pm 0.7763),\quad (-0.7050,0.7073,\pm 0.7763).
\] 
The corresponding minimizers of {\upshape(\ref{gcd:exm})} are
\[
(x_1\approx -1.0033,x_2\approx \pm 1.1011),\quad 
(x_1\approx  -1.0033,x_2\approx\pm 1.1011)
\] 
which are the same as in {\upshape\cite{Nie2008}}. 
The minimum is $r^* \approx 0.0643$. \hfill$\square$ }
\end{example}

\subsubsection{Constrained rational optimization}

We now give some numerical examples of minimizing of rational functions with
polynomial inequality constraints.  We first consider an example for which
$p(x)$ and $q(x)$ have common roots.
\begin{example}\cite{Nie2008}\label{ex::commonzeros2} 
{\upshape
Consider the following optimization
\begin{equation}
\begin{aligned}\label{exm:con:case:3}
\min\limits_{x\in\RR}\ r(x):=\frac{1+x}{(1-x^2)^2}\quad \text{s.t.}\ (1-x^2)^3\geq 0.
\end{aligned}
\end{equation}
As shown in \cite{Nie2008}, the global minimum $r^* =\frac{27}{32}\approx
0.8438$ and the minimizer $x^*=-\frac{1}{3}\approx-0.3333$. If the denominator
and numerator have common roots, SOS relaxation method proposed in
\cite{Nie2008} can not guarantee to converge to the minimum. 

Reformulating the above problem by homogenization, we get
\begin{equation}
\left\{
\begin{aligned}\label{exm:con:case:4}
\min\limits_{x_0,x_1\in\RR}&\ x^4_0+x_1x^3_0 \\
\text{s.t.}&\ x^4_0-2x_1^2x^2_0+x_1^4=1,\\
&\ x^6_0-3x^4_0x_1^2+3x^2_0x_1^4-x_1^6\geq 0,\ x_0\geq 0.
\end{aligned}\right.
\end{equation}
For relaxation order $N=7$, by the Jacobian SDP relaxation, we get the optimal
solution of {\upshape(\ref{exm:con:case:4})} $\tilde{x}^*\approx ( 1.0607,-0.3536)$ and the
minimum $s^*\approx 0.8437$. According to {\upshape(\ref{item::1})} in Theorem
\ref{thm::main2}, we find the minimizer of {\upshape(\ref{exm:con:case:3})}:\ $x^*\approx -0.3334$.
\hfill $\square$
}
\end{example}

We next consider Example \ref{ex::robinson} with some constraints. 
\begin{example}
{\upshape
Consider optimization
\begin{equation}\label{ex::robinson::cons}
\begin{aligned}
r^*:=\min\limits_{x\in S}&\ \frac{p(x_1,x_2)}{q(x_1,x_2)} =
\frac{x^6_1+x^6_2+3x^2_1x^2_2+1}{x^2_1(x^4_2+1)+x^2_2(x^4_1+1)+(x^4_1+x^4_2)}.\\
\end{aligned}
\end{equation}
\begin{enumerate}[\upshape (a)]
\item $S=\{(x_1,x_2)\in\RR^2\mid x^2_1+x^2_2\leq 1\}$. 
It is easy to check that $S$ is closed at $\infty$. By Theorem \ref{thm::main}, 
$r^*$ is equal to the optimum of the following optimization:
\begin{equation*}
\left\{
\begin{aligned}
s^*=\min\limits_{x_0,x_1,x_2\in\RR}&\ x^6_1+x^6_2+3x^2_1x^2_2x^2_0+x^6_0 \\
\text{s.t.}&\ x^2_1(x^4_2+x^4_0)+x^2_2(x^4_1+x^4_0)+x^2_0(x^4_1+x^4_2)=1,\\
&\ x^2_0-x^2_1-x^2_2\geq 0,\ x_0\geq 0.
\end{aligned}\right.
\end{equation*}
For relaxation order $N=7$, we get $r^*=s^*=1$ with 4 approximate minimizers:
\begin{equation*}
\begin{aligned}
&(0.8909,\pm 0.8909,-0.0000),&\quad (0.8909,-0.0000,\pm 0.8909),\\
\end{aligned}
\end{equation*}
which correspond to the exact minimizers: 
\[
(\frac{1}{\sqrt[6]{2}},\pm\frac{1}{\sqrt[6]{2}},0),\quad
(\frac{1}{\sqrt[6]{2}},0,\pm\frac{1}{\sqrt[6]{2}}).
\]
Then we get four minimizers of {\upshape(\ref{ex::robinson::cons})}: $(\pm1,0)$, $(0,\pm1)$.

\item $S=B(0,\sqrt{2})^{c}=\{(x_1,x_2)\in\RR^2\mid x^2_1+x^2_2\geq 2\}$.
$S$ is noncompact but closed at $\infty$. By Theorem \ref{thm::main}, we solve
the following equivalent optimization:
\begin{equation*}
\left\{
 \begin{aligned}
s^*:=\min\limits_{x_0,x_1,x_2\in\RR}&\ x^6_1+x^6_2+3x^2_1x^2_2x^2_0+x^6_0 \\
\text{s.t.}&\ x^2_1(x^4_2+x^4_0)+x^2_2(x^4_1+x^4_0)+x^2_0(x^4_1+x^4_2)=1,\\
&\ x^2_1+x^2_2- 2x^2_0\geq 0,\ x_0\geq 0.
\end{aligned}\right.
\end{equation*}
For relaxation order $N=7$, we get $r^*=s^*=1$ with $8$ approximate minimizers: 
\begin{equation*}
\begin{aligned}
&(0.0002,\pm 0.8909,\pm 0.8909),&\quad (0.7418,\pm 0.7419,\pm 0.7419),\\
\end{aligned}
\end{equation*}
which correspond to the exact minimizers:
\[
(0,\pm\frac{1}{\sqrt[6]{2}},\pm\frac{1}{\sqrt[6]{2}}),\quad
(\frac{1}{\sqrt[6]{6}},\pm\frac{1}{\sqrt[6]{6}},\pm\frac{1}{\sqrt[6]{6}}).
\]
The former solutions indicate that $r^*=1$ is also an asymptotic values at
$\infty$.  From the latter solutions, we get four  minimizers of
{\upshape(\ref{ex::robinson::cons})}: $(\pm1,\pm1)$.  \hfill $\square$
\end{enumerate} }
\end{example}

\vspace{10pt}
\noindent{\bfseries Acknowledgments:}
The research was partially supported by NSF grant DMS-0844775 and Research Fund for Doctoral Program of Higher Education of China grant 20114301120001.

\def\refname{\Large\bfseries References}
\bibliographystyle{plain}
\bibliography{fguo}

\end{document}